\documentclass[10pt,reqno]{amsart}

\usepackage[a4paper,left=33mm,right=33mm,
top=30mm,bottom=30mm,marginpar=25mm]{geometry}

\usepackage{amsmath}
\usepackage{amssymb}
\usepackage{amsthm}
\usepackage[latin1]{inputenc}
\usepackage{eurosym}
\usepackage[dvips]{graphics}
\usepackage{graphicx}
\usepackage{epsfig}
\usepackage{hyperref}
\usepackage{dsfont}

\usepackage[displaymath,mathlines]{lineno}

\allowdisplaybreaks

\usepackage{hyperref}

\usepackage{ifthen}
\usepackage{enumitem}

\providecommand{\varitem}{} 
\makeatletter
\newenvironment{hypotheses}[1]
 {\renewcommand\varitem[1]{\item[
 \textbf{#1\arabic{enumi}\rlap{$##1$}.}]%
    \edef\@currentlabel{#1\arabic{enumi}{$##1$}}}%
  \enumerate[label=\bf{({#1\arabic*})}, ref=#1\arabic*]}
 {\endenumerate}
\makeatother

\makeindex

\usepackage{fancyhdr}
\usepackage[us,24hr]{datetime}

\usepackage{scrtime}
\usepackage{comment}

\newcommand{\tr}{\operatorname{tr}}

\renewcommand{\div}{\operatorname{div}}

\newcommand{\Rr}{{\mathbb{R}}}

\newcommand{\Nn}{{\mathbb{N}}}

\newcommand{\Tt}{{\mathbb{T}}}

\newcommand{\Hh}{{\overline{H}}}

\newcommand{\Ll}{{\mathcal{L}}}

\newcommand{\Dd}{{\mathcal{D}}}

\newcommand{\Mm}{{\mathcal{M}}}

\newcommand{\epsi}{\varepsilon}

\def\dx{{\rm d}x}

\def\leq{\leqslant}
\def\geq{\geqslant}

\newcommand{\aev}{{a.e.}}
\def\bfm{\boldsymbol{m}}
\def\bfeta{\boldsymbol{\eta}}

\let\weakly\rightharpoonup
\def\weaklystar{\buildrel{\hskip-.6mm\star}\over\weakly}

\newcommand{\RR}{\mathbb{R}}
\newcommand{\NN}{\mathbb{N}}
\newcommand{\ZZ}{\mathbb{Z}}

\newcommand{\EE}{\mathbb{E}}

\newcommand{\intt}{\int_{\Tt^d}}
\newcommand{\fc}{\mathfrak{c}}

\numberwithin{equation}{section}

\newtheoremstyle{thmlemcorr}{10pt}{10pt}{\itshape}{}{\bfseries}{.}{10pt}{{\thmname{#1}\thmnumber{
#2}\thmnote{ (#3)}}}
\newtheoremstyle{thmlemcorr*}{10pt}{10pt}{\itshape}{}{\bfseries}{.}\newline{{\thmname{#1}\thmnumber{
#2}\thmnote{ (#3)}}}
\newtheoremstyle{defi}{10pt}{10pt}{\itshape}{}{\bfseries}{.}{10pt}{{\thmname{#1}\thmnumber{
#2}\thmnote{ (#3)}}}
\newtheoremstyle{remexample}{10pt}{10pt}{}{}{\bfseries}{.}{10pt}{{\thmname{#1}\thmnumber{
#2}\thmnote{ (#3)}}}
\newtheoremstyle{ass}{10pt}{10pt}{}{}{\bfseries}{.}{10pt}{{\thmname{#1}\thmnumber{
A#2}\thmnote{ (#3)}}}

\theoremstyle{thmlemcorr}
\newtheorem{theorem}{Theorem}
\numberwithin{theorem}{section}
\newtheorem{lemma}[theorem]{Lemma}
\newtheorem{corollary}[theorem]{Corollary}

\theoremstyle{thmlemcorr*}
\newtheorem{theorem*}{Theorem}
\newtheorem{lemma*}[theorem]{Lemma}
\newtheorem{corollary*}[theorem]{Corollary}
\newtheorem{proposition*}[theorem]{Proposition}
\newtheorem{problem*}[theorem]{Problem}
\newtheorem{conjecture*}[theorem]{Conjecture}

\theoremstyle{defi}

\theoremstyle{remexample}
\newtheorem{remark}[theorem]{Remark}

\newtheorem{teo}[theorem]{Theorem}
\newtheorem{lem}[theorem]{Lemma}
\newtheorem{pro}[theorem]{Proposition}

\theoremstyle{ass}

\usepackage{color}

\address{King Abdullah University
of Science
and Technology (KAUST), CEMSE Division \&
KAUST SRI, Center for
Uncertainty Quantification  in Computational
Science and Engineering,  Thuwal 23955-6900,
Saudi Arabia.}
\email[R.~Ferreira]{rita.ferreira@kaust.edu.sa}
\email[D.~Gomes]{diogo.gomes@kaust.edu.sa}

\thanks{
The authors were partially supported by KAUST baseline and start-up funds and by 
KAUST SRI, Center for Uncertainty Quantification in Computational Science and Engineering.}

\usepackage[normalem]{ulem}

\begin{document}

\title[Mean-Field Games through Variational
Inequalities]{Existence of Weak Solutions
to Stationary Mean-Field Games through Variational
Inequalities}
\author[\hfill Rita Ferreira and Diogo Gomes 
\hfill]{Rita Ferreira and Diogo Gomes}

\begin{abstract}
Here, we consider stationary monotone mean-field games (MFGs) and study the existence of weak solutions. First, we introduce a regularized problem that preserves the monotonicity. Next, using variational inequality techniques, we prove the existence of solutions to the regularized problem. Then, using Minty's method, we establish the existence of solutions for the original MFG. Finally, we examine the properties of these weak solutions in several examples. Our methods provide a general framework
to construct weak solutions to stationary MFGs 
 with local, nonlocal, or congestion terms.

\vspace{8pt}
%
\begin{center}\today\end{center}

\end{abstract}

\maketitle

\section{Introduction}

In the last years, mean-field games (MFGs) have become an active area of research in both the mathematics \cite{ll1,ll2,ll3}
and engineering communities \cite{Caines2, Caines1}.
In spite of substantial progress, many questions remain open.
Some of the more fundamental issues regard the existence and uniqueness of solutions. Various authors have attempted to answer these questions through explicit solutions and transformations
\cite{MR2928378, BP2, BP, Ge, MR2974160},
a priori estimates \cite{GMit,GPatVrt,  GP2,GPim1,
 GP13,GPim2,GPM1,GR, GM, GVrt2,
  PV15}, penalization ideas \cite{GPat2, GPat},
random-variable techniques
\cite{MR3045029, lacker1,GVrt,
lackerC, MR3113415},
weak and renormalized solutions \cite{porretta,
Por2015},
and variational methods \cite{Cd1, Cd2, cgbt, MR3358627}.
Here, we develop a new approach to investigate the existence of weak solutions to stationary MFGs using variational inequalities. 

We consider MFGs given
by the system of partial differential equations
{\setlength\arraycolsep{0.5pt}
\begin{equation}
\label{P}
\begin{cases}
\displaystyle -u -H(x, Du, D^2 u, m,h(\bfm))
=0,\\
m - \div(m D_p H(x, Du, D^2 u,m, h(\bfm)))
+ \big(m D_{M_{ij}} H (x, Du, D^2 u, m,h(\bfm)) \big)_{x_i
x_j} =1. 
\end{cases}
\end{equation}}%
For convenience,  we take periodic boundary conditions for \eqref{P} and work in the $d$-dimensional torus, $\Tt^d$, $d\in\Nn$. 
Moreover, \(i,j \in \{1, ..., d\}\) and we use the Einstein convention on repeated indices. In \eqref{P}, the Hamiltonian
{\setlength\arraycolsep{0.5pt}%
\begin{eqnarray}
\begin{aligned}\label{H}
H:\Tt^d \times \RR^d \times \RR^{d\times
d} \times \mathfrak{\mathbb{E} \times\RR} & \to\Rr \\
(x,p,M,m,\theta)
&\mapsto H(x,p,M,m,\theta) 
\end{aligned}
\end{eqnarray}}%
satisfies the regularity and growth assumptions 
discussed in Section \ref{Sect:assp}. The set $\mathbb{E}$ is either
 \(\RR^{+}\) or \(\RR^{+}_0\), and  \(h:\Mm_{ac}(\Tt^d) \to
 C(\Tt^d)\)  is a non-local operator also examined later. 
 Here, \(\Mm_{ac}(\Tt^d)\) stands for the space of  positive measures on \(\Tt^d\), absolutely continuous with respect to the Lebesgue measure. We use a boldface
font to denote elements of \(\Mm_{ac}(\Tt^d)\)
and represent their densities with the same non-boldface letter.

Our assumptions encompass a broad class of Hamiltonians that includes
 the following important examples:
 {\setlength\arraycolsep{0.5pt}%
        \begin{eqnarray}\label{H's}
        \begin{aligned}
        &H(x,p,M,m,\theta):= \bar H(x,p,m) - V(x) - g(m,\theta) - \sigma(x) \sum_{i=1}^d M_{ii},\\
        & h(\bfm):= c_1  \,\zeta * \bfm +c_2\, \zeta * \big( (\zeta*\bfm)^{\bar\alpha} \big),
        \end{aligned}
        \end{eqnarray}}%
 where 
 {\setlength\arraycolsep{0.5pt}%
        \begin{eqnarray}\label{H0's}\begin{aligned}
       \bar H(x,p,m) := a(x) (1 + |p|^2)^{\frac{\gamma}{2}} \quad \hbox{or} \quad \bar H(x,p,m) := a(x) \frac{|p|^2}{2 m^{\tau}},
        \end{aligned}
        \end{eqnarray}}%
 and 
 {\setlength\arraycolsep{0.5pt}%
        \begin{eqnarray}\label{g's}
        \begin{aligned}
        g(m,\theta):= m^\alpha + \theta
        \quad \hbox{or} \quad g(m,\theta):= \ln(m) + \theta,
        \end{aligned}
        \end{eqnarray}}%
 with $V,\, \sigma,\,\zeta,\,  a:\Tt^d \to \RR$ smooth, periodic functions such that  $\sigma\geq 0$, $\inf_{\Tt^d}
 \zeta \geq 0$, $\inf_{\Tt^d} a >0$, $c_1,\, c_2 \geq 0$,   $\gamma > 1$, $0 \leq \tau < 1$, and $\bar \alpha,\, \alpha >0$.

MFGs model problems with  large numbers of competing rational agents who seek to optimize an individual utility, see, for example, the lectures in
\cite{LCDF}.  In the time-dependent case, these games are given by a (time-dependent) Hamilton-Jacobi equation coupled with a transport or Fokker-Planck equation.  
The stationary case captures ergodic 
equilibria  and corresponds to the long-time limit of time-dependent MFGs \cite{CLLP, cllp13}. For up-to-date developments  on  MFGs,
 we refer the reader to the recent monographs \cite{bensoussan, GPV},
 the survey paper \cite{GS}, and the notes \cite{cardaliaguet}.
For numerical aspects, we recommend
\cite{achdou2013finite} and the references
therein.  

Uniformly elliptic MFGs and their corresponding
weak solutions were introduced in \cite{ll1}. 
The systematic study of the regularity 
theory for these MFGs was developed
in \cite{GPatVrt, GPM1,GM,  PV15}.
Those references establish the existence of classical solutions
for local MFGs with logarithmic nonlinearities  and power-like nonlinearities with certain growth conditions.  A particular stationary  congestion model was considered in \cite{GMit}.
Little is known about the general stationary
congestion problem.
Finally, we point out that some results
for time-dependent problems  relying on the
variational structure of MFGs  \cite{Cd1, Cd2, cgbt} and some weak solution methods \cite{porretta,
Por2015} may be extended
to the stationary case. However, to the best of
our knowledge, this has not been pursued in the literature.

In the MFGs literature, there are several gaps in the existence of solutions
that we try to address here. First, to obtain smooth solutions of local MFGs, 
 state-of-the-art methods \cite{PV15}
require growth conditions
in the non-linearities. These conditions seem to be of a technical
nature rather than of a fundamental nature. Second, regarding weak solutions, the uniformly parabolic case
with subquadratic or quadratic Hamiltonians
is well understood \cite{porretta, Por2015}.
In contrast, 
the degenerate parabolic case and  the 
uniformly parabolic with superquadratic Hamiltonians case
are  well understood only for variational problems  \cite{Cd1, Cd2, cgbt}.
Moreover, we expect
analogous results to hold for degenerate elliptic problems. 
Unfortunately, variational MFGs are
a restricted class
of problems
that is unstable under perturbations. For example, 
adding a small, non-local perturbation to a variational
MFG should not change the theory substantially, but 
it destroys the variational structure. Third,
apart from the preliminary results in \cite{GMit}
and the short-time problems in \cite{GVrt2, Graber2}, 
little is known about weak or classical solutions of congestion
models.

Here, we present a unified approach to studying
these problems and construct weak solutions based
on monotone operator methods. Monotonicity assumptions
are central in MFG theory and  are at the heart of the
uniqueness proof by Lasry and Lions \cite{LCDF}; also, the numerical methods in \cite{AFG} (for stationary MFGs) and  \cite{GJS2} (for
finite state MFGs) rely on monotonicity ideas. 
Monotone operator methods have several advantages. First, monotonicity is stable under
perturbations. Second, there is a well-developed
theory of weak solutions that, when combined with the 
Minty method, makes it possible to consider various limit
problems. Finally, our approach to MFGs through monotonicity
answers the earlier existence questions
and suggests new computational approaches. 

Next, we put forward the basic definitions. A  weak solution to \eqref{P} is a pair $(m,u)\in \mathcal{D}'
(\Tt^d) \times \mathcal{D}'(\Tt^d)$ with $m\geq 0$ 
that 
satisfies the variational inequality
\begin{eqnarray}\label{ws}
\begin{aligned}
\left\langle  \begin{bmatrix}\eta \\
v \\
\end{bmatrix}
 -  \begin{bmatrix}m \\
u \\
\end{bmatrix}, A \begin{bmatrix}\eta \\
v \\\end{bmatrix} \right\rangle_{\mathcal{D}'(\Tt^d) \times
\mathcal{D}'(\Tt^d), C^\infty(\Tt^d) \times
C^\infty(\Tt^d)} \geq 0
\end{aligned}
\end{eqnarray}
for all $(\eta, v) \in {C^\infty}(\Tt^d;\EE) \times
{C^\infty}(\Tt^d)$, where
\begin{eqnarray}\label{A}
\begin{aligned}
A \begin{bmatrix}\eta \\
v \\
\end{bmatrix}
:= \begin{bmatrix}\displaystyle -v- H(x, Dv, D^2 v, \eta, h(\bfeta)) \\
\eta - \div(\eta D_p H(x, Dv, D^2 v, \eta, h(\bfeta))) + \big(\eta
D_{M_{ij}} H (x, Dv, D^2 v, \eta,
h(\bfeta))\big)_{x_i x_j} -1  \\
\end{bmatrix}.
\end{aligned}
\end{eqnarray}%
Any solution to \eqref{P} is a weak solution. 
Moreover, suppose that $H$ is regular and let $(m,u)\in C^\infty(\Tt^d) \times C^\infty(\Tt^d)$ with $m>0$  a weak solution. A straightforward argument shows that $(m,u)$ solves \eqref{P}.
To see this, fix a pair 
$(\psi, \phi)\in C^\infty(\Tt^d) \times C^\infty(\Tt^d)$. Then, for all sufficiently small $\epsilon>0$, $m+\epsilon \psi>0$;  taking $(\eta,v)=(m+\epsilon \psi, u+\epsilon \phi)$ in \eqref{ws}, dividing by $\epsilon$, and letting $\epsilon\to 0^+$ yield
\begin{eqnarray*}
\begin{aligned}
\left\langle  \begin{bmatrix}\psi \\
\phi \\
\end{bmatrix}
 , A \begin{bmatrix}m \\
u \\
\end{bmatrix} \right\rangle_{\mathcal{D}'(\Tt^d) \times
\mathcal{D}'(\Tt^d), C^\infty(\Tt^d) \times
C^\infty(\Tt^d)} \geq 0.
\end{aligned}
\end{eqnarray*}
Because $\phi$ and $\psi$ are arbitrary, we conclude
that $(m,u)$ solves \eqref{P}.
 
Our goal is to prove the existence of weak solutions and  to
study their properties. Our main result is the following.
\begin{teo}\label{Thm:main}
Suppose that Assumptions \eqref{h}, \eqref{g},
\eqref{G:g ei}, and \eqref{bddH}--\eqref{H:mon} are satisfied. Then, there exists a weak solution, $(m,u)\in
\mathcal{D}'(\Tt^d) \times \mathcal{D}'(\Tt^d)$,
with $m\geq 0$   to \eqref{P}. Moreover, \((\bfm,u) \in \Mm_{ac}(\Tt^d) \times W^{1,\gamma} (\Tt^d)  \) and \(\int_{\Tt^d} m \,\dx = 1\).
\end{teo}

This foregoing theorem gives the existence of solutions
for a minimal set of assumptions and  
is a substantial improvement on prior results.  
First, the theorem is valid for
degenerate elliptic MFGs. These are technically challenging
since  various analytical techniques for first-order
MFGs do not apply,
and the regularizing effects due to ellipticity are mild to non-existent. In this case,
prior results
only apply to problems with a variational structure.  
Second, the theorem holds
for congestion MFGs. These were studied 
before for a particular problem
in \cite{GMit}; little is known about the existence
of solutions for general stationary problems. In the time-dependent
setting, only the short-time problem
has been considered in the literature
\cite{GVrt2, Graber2}.

While Theorem~\ref{Thm:main} gives the existence of solutions
 for MFGs, these solutions have low regularity. In Section~\ref{SOE}, we consider the degenerate
diffusion case and investigate further properties
of the  weak solutions constructed in Theorem~\ref{Thm:main}.
Under appropriate
conditions, we prove that such  solutions are,
in the sense of distributions, subsolutions of the first equation in \eqref{P}
(Proposition~\ref{subsolprop}), relaxed solutions of the
second equation in  \eqref{P}
(Proposition~\ref{wfp}), and  relaxed supersolutions of
the first equation in \eqref{P}
(Proposition~\ref{supersolprop}).
The quadratic case is examined in detail
in  Section~\ref{Sect:eaa},
where we establish higher integrability and Sobolev
estimates for \(m\). Moreover, under appropriate
assumptions, we prove that the second equation
in \eqref{P} is satisfied pointwise in \(\Tt^d\)
(Theorem~\ref{Thm:quad})
 and that the first equation
in \eqref{P} is satisfied pointwise in the set
where \(m\) is positive (Corollary~\ref{Cor:quad}).

This paper is structured as follows. 
Our main assumptions  are discussed in Section~\ref{Sect:assp}. 
Then, in Section \ref{Sect:pert}, we begin the study
of \eqref{P}. First, we introduce a regularized
problem that involves two positive, small parameters.
Next, we prove a priori estimates for solutions
to the regularized problem. 
Combining these estimates with a continuation argument,
we show the existence of solutions to the regularized problem. In Section~\ref{Sect:pmthm}, we  establish further uniform estimates
with respect to the  parameters that allow
us to pass  the regularized problem
to the limit as these parameters tend to zero.
Then,  this limiting
procedure enables us to prove Theorem~\ref{Thm:main}
by using the Minty device.
Finally, as mentioned before,   Sections~\ref{SOE} and \ref{Sect:eaa} are devoted to establishing
further properties of weak solutions for 
specific Hamiltonians. We conclude this paper
in Section~\ref{final} with some final remarks.

\section{Main assumptions}
\label{Sect:assp}

Here, we discuss the main assumptions of this
paper.
In our choice of assumptions, we have attempted to balance generality with brevity and clarity.
Naturally, this requires some compromises in illustrating the core ideas, and several extensions and variations of our results are possible though not included
in our discussion. Nevertheless, our assumptions encompass a broad range of examples, including variational mean-field games with power and logarithmic nonlinearities, congestion problems, elliptic, degenerate elliptic, and first-order Hamiltonians with both sub- and superquadratic growth.

We recall that \(\Tt^d\) represents the \(d\)-dimensional torus, which we identify with the quotient space \(\RR^d/\ZZ^d\). As before,  \(\Mm_{ac}(\Tt^d)\)
is the space of  positive measures
on \(\Tt^d\), absolutely continuous with respect
to the Lebesgue measure. Moreover, elements of \(\Mm_{ac}(\Tt^d)\) are denoted with boldface
font while their densities are represented
with the same letter in a non-boldface font.
For instance, if $\bfm\in \Mm_{ac}(\Tt^d)$, then $m \in L^1(\Tt^d)$ represents the corresponding density function.
Similarly,
if $m\in L^1(\Tt^d)$ is such that $m \geq 0$ \aev\ in $\Tt^d$ and \({\mathcal{L}^d}\) represents the \(d\)-dimensional Lebesgue measure, then $\bfm:= m{\mathcal{L}^d}_{\lfloor \Tt^d}$ denotes the corresponding measure in \(\Mm_{ac}(\Tt^d)\).
We also recall that the set
$\mathbb{E}$ is either
 \(\RR^{+}\) or \(\RR^{+}_0\).

To avoid longer expressions, we  often omit
some explicit dependencies on the space variable.
For example, we  write \(H(x, Du, D^2
u, m, h(\bfm))\) in place of \(H(x, Du(x), D^2
u(x), m(x), h(\bfm)(x))\); similarly, we write
\(g(x, m, h(\bfm))\) in place of \(g(x, m(x), h(\bfm)(x))\),
and so on.
 
In what follows,
{\setlength\arraycolsep{0.5pt}%
\begin{eqnarray*}
\begin{aligned}
h: \Mm_{ac}(\Tt^d)
 &\to C(\Tt^d)\\
 \bfm &\mapsto h(\bfm)
 \end{aligned}
\end{eqnarray*}}%
is a (possibly nonlinear) operator. In most
of our statements, we suppose that:
\begin{hypotheses}{h}
\item \label{h}
For all  $\kappa > \frac{d}{2}+1$:
\begin{enumerate}
    \item \( \{h(\bfm)\!: \, m \in
W^{\kappa,2}(\Tt^d)\} \subset W^{\kappa,2}
(\Tt^d) \);
    \item \(m\in  W^{\kappa,2}(\Tt^d) \mapsto h(\bfm)
\in  W^{\kappa,2}(\Tt^d)\)
  defines  a Fr\'echet differentiable map.
  \newcounter{hyph}
  \setcounter{hyph}{\arabic{enumi}}
  \end{enumerate}
\end{hypotheses}

If \(h\) satisfies the Assumption~\eqref{h}, then  for all $m_0 \in W^{\kappa,2}(\Tt^d)$, there exists a bounded linear operator, \(\mathfrak{H}_{m_0} \in \mathcal{L} (W^{\kappa,2}(\Tt^d) ; W^{\kappa,2}(\Tt^d)) \), such that, for all $m \in  W^{\kappa,2}(\Tt^d)$,
{\setlength\arraycolsep{0.5pt}
\begin{eqnarray}\label{hFD}
\begin{aligned}
&\Vert h(\bfm) - h(\bfm_0)\Vert_{W^{\kappa,2}(\Tt^d)} \\
&\qquad\leq \Vert \mathfrak{H}_{m_0}\Vert_{\mathcal{L}(W^{\kappa,2}(\Tt^d); W^{\kappa,2}(\Tt^d))} \Vert m - m_0
\Vert_{W^{\kappa,2}(\Tt^d)} + o\big(\Vert m - m_0
\Vert_{W^{\kappa,2}(\Tt^d)} \big).
\end{aligned}
\end{eqnarray}}%
Moreover, taking $m_0 =0$ in \eqref{hFD}, we get
{\setlength\arraycolsep{0.5pt}
\begin{eqnarray}\label{on h(m)}
\begin{aligned}
\Vert h(\bfm)\Vert_{W^{\kappa,2}(\Tt^d)} \leq C_0 \big( 1+\Vert m\Vert_{W^{\kappa,2}(\Tt^d)} \big)
\end{aligned}
\end{eqnarray}}%
for some positive constant $C_0= C_0
\big (\kappa,\Tt^d, \Vert \mathfrak{H}_0\Vert_{\mathcal{L}
(W^{\kappa,2}(\Tt^d);
W^{\kappa,2}(\Tt^d))} ,\Vert \mathfrak{H}(\boldsymbol{0})
\Vert_{W^{\kappa,2}(\Tt^d)} \big)$.

Examples of operators  satisfying the previous
assumption are those in \eqref{H's}. 

To describe the behavior of the Hamiltonian \eqref{H} in the variables \((x,m,\theta)\), we introduce a continuous function, $g:\Tt^d \times \EE \times\RR\to\RR$, and the following  assumptions:
 \medskip
\begin{hypotheses}{g}
\item \label{g}
  \begin{enumerate}
    \item \label{G:sumg}  There exist functions,
\(g_1\) and \(g_2\), such that, for all $(x, m,\theta)
        \in \Tt^d \times  \EE \times \RR$,
        {\setlength\arraycolsep{0.5pt}%
        \begin{eqnarray*}
        \begin{aligned}
        g(x,m, \theta) = g_1(x,m, \theta) +         g_2(x,m, \theta)
        \end{aligned}
        \end{eqnarray*}}%
         and, for all \(\bfm\in  \Mm_{ac}(\Tt^d) \), \[g_2(x,m(x), h(\bfm)(x)) \geq 0.\]
    
        \item \label{G:g+} 
        There exists a constant, $c>0$,              such that, for all $\bfm  \in \Mm_{ac}(\Tt^d)$,
        {\setlength\arraycolsep{0.5pt}%
        \begin{eqnarray*}
        \begin{aligned}
        \qquad \int_{\Tt^d} g_2(x ,m,h(\bfm))\,\dx
        \leq c \bigg(1+\int_{\Tt^d} m\,\dx\bigg).
        \end{aligned}
        \end{eqnarray*}}%
    \item \label{G:g ub}
        For all $\delta>0$, there exists            a constant, $C_\delta >0$, such that,
        for all $\bfm  \in \Mm_{ac}(\Tt^d)$,
        {\setlength\arraycolsep{0.5pt}%
        \begin{eqnarray*}
        \begin{aligned}
        \max\bigg\{ \int_{\Tt^d} g_1(x,m,h(\bfm))         \,\dx, \int_{\Tt^d} m\,\dx \bigg\}         \leq \delta \int_{\Tt^d}
        mg_1(x,m,h(\bfm))\,\dx + C_\delta.
        \end{aligned}
        \end{eqnarray*}}%
     
    \item \label{G:g lb}
        There exists $R\geq 0$ such that,           for all $\bfm \in \Mm_{ac}(\Tt^d)$,
        {\setlength\arraycolsep{0.5pt}%
        \begin{eqnarray*}
        \begin{aligned}
        \int_{\Tt^d} mg_1(x,m,h(\bfm))\,\dx         \geq - R.
        \end{aligned}
        \end{eqnarray*}}%
  \end{enumerate}
 
\item  \label{G:g ei}
 If  \((\bfm_j)_{j\in\NN} \subset  \Mm_{ac}(\Tt^d)\)
is such that \[\sup_{j\in\NN} \int_{\Tt^d}
m_jg_1(x,m_j,h(\bfm_j))\,\dx < +\infty,\]
then there exists a  subsequence 
of  $(m_j)_{j\in\NN}$ that  converges weakly
in $L^1(\Tt^d)$.
\end{hypotheses} 
\medskip

The functions  $g$  in \eqref{g's} with \(h\)
as in \eqref{H's} satisfy
Assumptions~\eqref{g} and \eqref{G:g ei} for \(g_1(x,m,\theta):=
m^\alpha\) or \(g_1(x,m,\theta) :=
\ln(m)\), and \(g_2(x,m,\theta):=
\theta\).

Next, we enumerate  hypotheses on the Hamiltonian.
Recall that $M_1:M_2=\tr (M_1 M_2)$ whenever \(M_1\)
and \(M_2\) are  symmetric matrices. \medskip

\begin{hypotheses}{H}
\item \label{bddH}
There exist constants, $\gamma >1$, $0 \leq \tau < 1$, \(C_1>0\), \(C_2>0\), \(C_3>0\), and \(C_4>0\),
and a  continuous 
function, $g:\Tt^d \times \EE \times\RR\to\RR$,  such that:
        \begin{enumerate}
        \item  \label{H:ae lb}
        for all $(x,p,M,m,\theta) \in        
     \Tt^d \times \RR^d \times \RR^{d\times
        d} \times\EE \times \RR$,
        {\setlength\arraycolsep{0.5pt}%
        \begin{eqnarray*}
        &&-H(x,p,M,m,\theta) + D_p H(x,p,M,m,\theta)         \cdot p + D_MH(x,p,M,m,\theta) : M         \\ 
        &&\qquad\geq\frac{1}{C_1}   m^{-\tau} |p|^\gamma         +C_2g(x,m,\theta)-  C_1;
        \end{eqnarray*}}%
        \item \label{H:int b}
        for all $(\bfm,u) \in      \mathcal{M}_{ac}(\Tt^d)
        \times W^{1, \gamma}(\Tt^d)$ with         $D^2u$ a measurable function,
        {\setlength\arraycolsep{0.5pt}%
        \begin{eqnarray*}
        \begin{aligned}
        & - C_3\int_{\Tt^d}  g(x,m,         h(\bfm))\,\dx  +\frac{1}{C_1}    \intt |Du|^\gamma m^{-\tau}\,\dx          - C_1 \\&\quad\leq \int_{\Tt^d}         H(x, Du, D^2 u, m, h(\bfm))\,\dx         \\  &\quad\leq - C_4\int_{\Tt^d}  g(x,m,      
  h(\bfm))\,\dx + C_1\bigg( 1+\int_{\Tt^d}          |Du|^\gamma         m^{-\tau} \,\dx \bigg). \end{aligned}
        \end{eqnarray*}}%
        \end{enumerate}

\item \label{H:smooth}
The map $(x,p,M,m,\theta) \mapsto H(x,p,M,m,
\theta) $  from \(\Tt^d \times \RR^d \times \RR^{d\times
d} \times \EE\times \RR\) to \(\RR\) is  of class $C^\infty$.

\item \label{H:mon}
The operator, $A$, defined in \eqref{A} is monotone
with respect to the $L^2 \times L^2$-inner product. More precisely,
for all  $(\eta, v), \, (\bar\eta, \bar
v)  \in {C^\infty}(\Tt^d; \EE) \times
{C^\infty}(\Tt^d)$, we have 
{\setlength\arraycolsep{0.5pt}%
\begin{eqnarray*}
\begin{aligned}
\left( A \begin{bmatrix}\eta \\
v \\
\end{bmatrix} -  A \begin{bmatrix} \bar\eta
\\
\bar v \\
\end{bmatrix}, \begin{bmatrix}\eta \\
v \\
\end{bmatrix}
 -  \begin{bmatrix} \bar \eta \\
\bar v \\
\end{bmatrix} \right)_{L^2(\Tt^d) \times L^2(\Tt^d)}
\geq 0.
\end{aligned}
\end{eqnarray*}}%
\newcounter{hypH}
  \setcounter{hypH}{\arabic{enumi}}
\end{hypotheses}
\medskip

Examples of Hamiltonians satisfying 
\eqref{bddH}--\eqref{H:mon} are
given by \eqref{H's}--\eqref{g's}.
Virtually all nonlinear
Hamiltonians in the MFGs literature satisfy these
assumptions, and 
Theorem \ref{Thm:main} holds under them.

As mentioned in the Introduction, we also  investigate
 further properties of the weak solutions
constructed in Theorem~\ref{Thm:main}. This study
requires
 additional assumptions on the Hamiltonian.
 The precise statements of the additional assumptions  are given in Section~\ref{SOE}
(see Assumption~\eqref{dege}) and in
Section~\ref{Sect:eaa} (see Assumption~\eqref{CC}).

\section{Regularized mean-field game}\label{Sect:pert}

To construct solutions to \eqref{P}, we introduce the regularized problem
\begin{equation}
\label{RP}
\begin{bmatrix} \displaystyle-u- H+\epsi_1(m
+\Delta^{2p}m) +\epsi_2(m -\Delta m)+ 
\beta_{\epsi_1}(m)\\
m - \div(m D_p H) + \big(m
D_{M_{ij}} H\big)_{x_i x_j} -1 + \epsi_1(u
+\Delta^{2p}u) +\epsi_2(u -\Delta u)
\end{bmatrix}
=
\begin{bmatrix} 
0\\
0
\end{bmatrix}, 
\end{equation}
where $\epsi_1,\, \epsi_2>0$,  $p\in \Nn$ satisfies
$2p-4>\frac{d}{2}+1$, and 
$\beta_{\epsi_1}:(0,+\infty)\to (-\infty,0]$ is a non-decreasing, \(C^\infty\) function satisfying $\beta_{\epsi_1}(s)=0$ if $s\geq \epsi_1$ and $\beta_{\epsi_1}(s) = -\frac{1}{s^q}$ if $0<s \leq \frac{\epsi_1}{2}$, where \(q>d\).
Moreover, the function
$H$
(and, analogously, $D_p H$ and $
D_{M_{ij}} H$) is 
identified with the map
\[
x\mapsto H(x, Du(x), D^2 u(x), m(x),h(\bfm)(x)).\]
 Here, we prove the following result.

\begin{pro}
\label{RPSol}
Suppose that Assumptions \eqref{h}, \eqref{g}, and \eqref{bddH}--\eqref{H:mon} hold. Then, for all $\epsi_1,\, \epsi_2 \in (0,1)$,  \eqref{RP} admits a $C^\infty$ 
solution $(m,u)$ with $\inf_{\Tt^d} m>0$. 
\end{pro}
We postpone the proof of this proposition to the end of this section
because we need to establish  certain 
 a priori bounds first.
These bounds rely on two main effects: 
the term $(m
+\Delta^{2p}m)$
gives regularity; the  
penalization term $\beta_{\epsi_1}$ 
forces $m$ to be strictly positive.
After proving these bounds, 
we use the continuation method to show the existence of a solution. 

\subsection{The perturbed problem} \label{tpp}

Fix $\epsi_1,\, \epsi_2 \in (0,1)$, $\mu\in[0,1]$,  $q>d$, and  $p\in \Nn$  such that $2p-4>\frac{d}{2}+1$.
Assume that \(H\) satisfies Assumption~\eqref{bddH},
and define
{\setlength\arraycolsep{0.5pt}%
        \begin{eqnarray*}
                \begin{aligned}
                        &H_\mu (x,p,M,m, \theta):= (1-\mu)H(x,p,M,m, \theta) +\mu \widetilde H(p,m),
                \end{aligned}
        \end{eqnarray*}}%
where
\begin{equation*}
\widetilde H(p,m):= \frac{(1+|p|^2)^{\frac{\gamma}{2}}}{\gamma
m^\tau}-m.
\end{equation*}
Set $\epsi:=(\epsi_1, \epsi_2)$, and define the operator $A^\epsi_\mu:D(A^\epsi_\mu) \subset L^2(\Tt^d)\times L^2(\Tt^d) \to  L^2(\Tt^d)\times L^2(\Tt^d)$, where  
\begin{eqnarray}\label{domAe}
\begin{aligned}
D(A^\epsi_\mu):=\left\{ (m,u)\in W^{4p,2}(\Tt^d) \times W^{4p,2}(\Tt^d)\!: \,\, \inf_{\Tt^d}m>0\right\},
\end{aligned}
\end{eqnarray}
by
\begin{eqnarray}\label{Ae}
\begin{aligned}
A^\epsi_\mu\begin{bmatrix}m \\
u \\
\end{bmatrix}
:= \displaystyle \begin{bmatrix} \displaystyle-u- H_\mu+\epsi_1(m
+\Delta^{2p}m) +\epsi_2(m -\Delta m)+ 
\beta_{\epsi_1}(m)\\
m - \div(m D_p H_\mu) + \big(m
D_{M_{ij}} H_\mu\big)_{x_i x_j} -1 + \epsi_1(u
+\Delta^{2p}u) +\epsi_2(u -\Delta u)
\end{bmatrix}\! \,
\end{aligned}
\end{eqnarray}
for $(m,u)\in D(A^\epsi_\mu)$. As before, the function $H_\mu$
(and, analogously, $D_p H_\mu$ and $
D_{M_{ij}} H_\mu$) is 
identified with the map
\[
x\mapsto H_\mu(x, Du(x), D^2 u(x), m(x),h(\bfm)(x)).
\]

The parameter \(\epsi_1\) controls the main regularizing effects; 
the parameter \(\epsi_2\) is a technical device
that is used later to   improve the regularity of the weak solutions. 
\begin{remark}\label{Rmk: oH}
Observe that  $\widetilde H$ satisfies Assumptions  \eqref{H:smooth} and \eqref{H:mon}. Moreover, if \(\tilde g_1(x,m,\theta):=m\) and   \(\tilde g_2(x,m,\theta):=0\), then \(\tilde g := \tilde
g_1 + \tilde g_2\) satisfies
Assumptions~\eqref{g} and \eqref{G:g ei}. Furthermore,
Assumption~\eqref{bddH} is also satisfied, where
\(C_1\) is replaced by some positive constant
that depends on \(\gamma\), \(\widetilde C_1= \widetilde C_1(\gamma) \), and with \(C_2=
C_3 = C_4 =1\). Without loss of generality, we
may assume that \(\widetilde C_1 = C_1\); otherwise,
we simply replace both \(C_1\) and \(\widetilde
C_1\) by \(\max\{C_1, \widetilde C_1\}\) and relabel
the constants conveniently. \end{remark}

Recall  that $W^{k,2}(\Tt^d)$ is an algebra for $k>\frac{d}{2}$. Moreover, 
 since
 $2p-4>\frac{d}{2} +1$, 
 Morrey's
theorem yields that the embedding $W^{2p-4,2}(\Tt^d)\hookrightarrow
C^\lambda(\Tt^d)$ is continuous for some
$\lambda\in (0,1)$. In particular,
there exists a positive constant, \(C\),  depending
only on $p$, $d$, and $\lambda$, such
that, for all $w\in W^{2p-4,2}(\Tt^d)$,
\begin{eqnarray}\label{SobEmb}
\begin{aligned}
\Vert w\Vert_{L^\infty(\Tt^d)}\leq C \Vert
w\Vert_{W^{2p-4,2}(\Tt^d)}.
\end{aligned}
\end{eqnarray}
 
Finally, consider the problem of finding $(m^\epsi_\mu,u^\epsi_\mu)\in D(A^\epsi_\mu)$ satisfying the  perturbed mean-field game
\begin{eqnarray}\label{Pe}
\begin{aligned}
A^\epsi_\mu\begin{bmatrix}m^\epsi_\mu \\
u^\epsi_\mu \\
\end{bmatrix} = \begin{bmatrix}0 \\
0 \\
\end{bmatrix}.
\end{aligned}
\end{eqnarray}
Note that
\eqref{Pe}  with \(\mu=0\) reduces to \eqref{RP}.

We set
\begin{eqnarray}\label{Lambda}
\Lambda^\epsi:=\{\mu\in [0,1]\!:\,\, \exists\, (m^\epsi_\mu,u^\epsi_\mu)\in
 D(A^\epsi_\mu)  \cap{C^\infty}(\Tt^d) \times  {C^\infty}(\Tt^d)  \hbox{ solving \eqref{Pe}}\}.
\end{eqnarray}
Our goal is to establish the equality
\begin{eqnarray}\label{Lambda=01}
\begin{aligned}
\Lambda^\epsi=[0,1].
\end{aligned}
\end{eqnarray}
To this end,
next, 
we begin by proving a priori estimates for classical solutions of \eqref{Pe}. These are essential to the application of
the continuation method. In Subsection \ref{Subs:L=[0,1]}, we
use this method to show that $\Lambda^\epsi$ is a nonempty set, relatively closed and open in $[0,1]$. Consequently, \eqref{Lambda=01} holds.

\subsection{A priori estimates for classical
solutions to the perturbed mean-field game}\label{Subs:ape}

We start by establishing preliminary a priori estimates for classical solutions of \eqref{Pe}. 
These estimates involve constants
whose main feature is their independence of any
particular choice of solutions to problem \eqref{Pe}.
To simplify the notation, we introduce some nomenclature
to specify these constants. We say that
\textit{a constant depends only on the problem
data} if
it is a function of the dimension, \(d\),  
 of \(p\), and of
the constants in  Assumptions~\eqref{h}, \eqref{g},
 and \eqref{bddH}.
We say that 
\textit{a constant depends only on the problem
data and on \(\epsi\)} if it is a function of
\(d\),  
 of \(p\),  of
the constants in  Assumptions~\eqref{h}, \eqref{g},  and \eqref{bddH}, of \(\lambda\),
of \(\epsi_1\), of \(\epsi_2\), and of the \(L^\infty\)-estimates
of \(\beta_{\epsi_1}
(\cdot)\) and its derivatives on a compact subset
of \((0,\infty)\) that depends only on \(\epsi_1\),
on \(\epsi_2\),
and on a constant that
depends only on the problem
data.
We stress that these constants do not depend 
  on
the choice of solutions to \eqref{RP} nor on \(\mu\).

\begin{lem}\label{Lem:ape}
Suppose that Assumptions \eqref{g} and \eqref{bddH} hold. Let
\(\epsi \in
(0,1)^2\)  and \(\mu \in [0,1]\). Assume that
 $(m^\epsi_\mu,u^\epsi_\mu)\in
C^\infty(\Tt^d)\times C^\infty(\Tt^d)$ solves \eqref{Pe}
 in $\Tt^d$.
 Suppose that 
 $m^\epsi_\mu>0$ in $\Tt^d$.
 Then, there exists a positive constant, $C$, that depends only on the problem
data such that
\begin{eqnarray}
&&\int_{\Tt^d}
|Du^\epsi_\mu|^\gamma (m^\epsi_\mu )^{1-\tau}\dx +  \int_{\Tt^d}
|Du^\epsi_\mu|^\gamma (m^\epsi_\mu )^{-\tau} \, \dx+ \int_{\Tt^d}
-\beta_{\epsi_1}(m^\epsi_\mu)\,\dx     \nonumber\\ 
 &&\qquad+ \int_{\Tt^d}
\beta_{\epsi_1}
(m^\epsi_\mu)(m^\epsi_\mu-\epsi_1) \,\dx+(1-\mu) \int_{\Tt^d} m^\epsi_\mu g_2(x,m^\epsi_\mu,h(\bfm^\epsi_\mu))\,\dx     \nonumber\\ 
 &&\qquad+\epsi_1 \int_{\Tt^d} \left[(u^\epsi_\mu)^2
+ (\Delta^p u^\epsi_\mu)^2 + (m^\epsi_\mu)^2
+ (\Delta^p m^\epsi_\mu)^2\right]\dx\qquad \nonumber\\
&&\qquad\displaystyle +\epsi_2
\int_{\Tt^d} \big((u^\epsi_\mu)^2 + |D
u^\epsi_\mu|^2 + (m^\epsi_\mu)^2
+ |D m^\epsi_\mu|^2 
\big) \,\dx \leq C. \label{ape1}
\end{eqnarray}
\end{lem}

\begin{proof}
In this proof, 
to simplify the notation, we drop the dependence on $\epsi$ and $\mu$ of $\bfm^\epsi_\mu$, $m^\epsi_\mu,$ and $u^\epsi_\mu$. Accordingly, we simply write $\bfm$, $m,$ and $u$ in place of $\bfm^\epsi_\mu$, $m^\epsi_\mu,$ and $u^\epsi_\mu$, respectively.
 
By the statement of the lemma, the two following equalities hold  in $\Tt^d$:
\begin{eqnarray}
&&\displaystyle -u - H_\mu + \epsi_1(m +\Delta^{2p}m) + \epsi_2(m -\Delta m) + \beta_{\epsi_1}(m)=0, \label{eqnme}\\
&&\displaystyle m - \div(m D_p H_\mu) + (m
D_{M_{ij}} H_\mu)_{x_i x_j} -1  + \epsi_1(u +\Delta^{2p}u) + \epsi_2(u -\Delta u)= 0,\qquad \label{eqnue}
\end{eqnarray}
where $H_\mu$, $D_p
H_\mu,$ and $
D_{M_{ij}} H_\mu$ are computed at $(x, Du(x), D^2
u(x),m(x), h( \bfm)(x))$.

We multiply  \eqref{eqnme} by $(m-\epsi_1-1)$ 
and \eqref{eqnue} by $u$. 
Next, adding the resulting expressions, integrating 
 over $\Tt^d$, and integrating by parts yield
{\setlength\arraycolsep{0.5pt}
\begin{eqnarray}
&&\displaystyle\int_{\Tt^d} m  \big(- H_\mu +D_p H_\mu\cdot Du + D_M H_\mu : D^2u\big)\,\dx +(\epsi_1 +1)\int_{\Tt^d} H_\mu \,\dx+ \int_{\Tt^d}
-\beta_{\epsi_1}(m)\,\dx \nonumber\\
 &&\quad+ \int_{\Tt^d}
\beta_{\epsi_1}
(m)(m-\epsi_1) \,\dx+\epsi_1 \int_{\Tt^d} \left[u^2
+ (\Delta^p u)^2 + m^2
+ (\Delta^p m)^2 
\right]\dx\qquad \nonumber\\
&&\quad\displaystyle +\epsi_2
\int_{\Tt^d} \big(u^2 + |D
u|^2 + m^2
+ |D m|^2 
\big) \,\dx \nonumber\\
&&  \qquad = -\epsi_1 \int_{\Tt^d} u\,\dx + (\epsi_1 +\epsi_2)(\epsi_1 +1) \int_{\Tt^d}
m\,\dx   \leq \frac{\epsi_1}{2} \int_{\Tt^d} u^2\,\dx + \frac{\epsi_1 +\epsi_2}{2} \int_{\Tt^d}
m^2\,\dx + \frac{5}{2}. \nonumber\\
&& \label{onape11}
\end{eqnarray}}%

Let \(c\), \(C_1\), \(C_2\), and \(C_3\) be as in Assumptions~\eqref{G:g+}
and \eqref{bddH} and as in Remark~\ref{Rmk: oH}; let
\( \mathfrak{c}:= \max\{C_1, 
C_3\} \), and set
{\setlength\arraycolsep{0.5pt}%
\begin{eqnarray*}
F^\epsi_\mu :=&& \frac 1 \fc \int_{\Tt^d}
|Du|^\gamma m^{1-\tau}\dx
+\frac 1 \fc\int_{\Tt^d}
|Du|^\gamma m^{-\tau}
\, \dx \\ 
&&+(1-\mu)C_2 \int_{\Tt^d} m g_2(x,m,h(\bfm))\,\dx
+ \int_{\Tt^d}
-\beta_{\epsi_1}(m)\,\dx     \nonumber\\ 
 &&+ \int_{\Tt^d}
\beta_{\epsi_1}
(m)(m-\epsi_1) \,\dx  
+\frac{\epsi_1}{2} \int_{\Tt^d} \left[ u^2
+ (\Delta^p u)^2 + m^2
+ (\Delta^p m)^2 
\right] \dx\qquad \nonumber\\
&&\displaystyle +\frac{\epsi_2}{2}
\int_{\Tt^d} \big(u^2 + |D
u|^2 + m^2
+ |D m|^2 
\big) \,\dx. 
\end{eqnarray*}}%
Due to \eqref{G:g+} and the definition of $\beta_{\epsi_1}(\cdot)$,  the integral terms defining $F^\epsi_\mu$  are nonnegative.
 
Using in \eqref{onape11} the Assumptions \eqref{H:ae lb} and \eqref{H:int b} for \(H\) and \(\widetilde
H\),
it follows that
{\setlength\arraycolsep{0.5pt}
\begin{eqnarray}
&&F^\epsi_\mu + (1-\mu)C_2  \int_{\Tt^d} m g_1(x,m,h(\bfm))\,\dx
+ \mu \intt m^2\,\dx \nonumber \\
&&\qquad\leq  \frac{5}{2} + 2\fc+(\fc+2)  
 \int_{\Tt^d} m\,\dx+ 2\fc  \int_{\Tt^d}  g_2(x,m,h(\bfm))\,\dx  \nonumber\\
 &&\qquad\qquad + C_3(\epsi_1 +1)(1-\mu)  \int_{\Tt^d}  g_1(x,m,h(\bfm))\,\dx \nonumber\\
&&\qquad\leq  \fc_1+ \fc_2 \mu \int_{\Tt^d} m\,\dx + \fc_2 (1-\mu) \int_{\Tt^d} m\,\dx+ \fc_3(1-\mu)\int_{\Tt^d}  g_1(x,m,h(\bfm))\,\dx,\qquad\label{onape12}
\end{eqnarray}}%
where $\fc_1 := \frac{5}{2}  
+ 2\fc + 2c\fc$, $\fc_2:=   
\fc+2+ 2c\fc$, and $\fc_3
\in \{C_3, 2C_3\}$, and where we also used \eqref{G:sumg} and \eqref{G:g+}. Using now \eqref{G:g ub}   with $\delta := \frac{C_2}{2(\fc_2+ \fc_3)}$, we conclude that
{\setlength\arraycolsep{0.5pt}
\begin{eqnarray}\label{onape13}
\begin{aligned}
F^\epsi_\mu  + \frac{(1-\mu)C_2}{2}   \int_{\Tt^d} m g_1(x,m,h(\bfm))\,\dx+\frac\mu 2\intt m^2\,\dx \leq \fc_1+ \frac{\fc_2^2}{2}+ (\fc_2+ \fc_3)C_{\delta} .\qquad
\end{aligned}
\end{eqnarray}}%
Finally, by \eqref{G:g lb}, we obtain
{\setlength\arraycolsep{0.5pt}%
\begin{eqnarray*}
\begin{aligned}
0\leq F^\epsi_\mu \leq \fc_1+ \frac{\fc_2^2}{2}+ (\fc_2+ \fc_3)C_{\delta}
+ \frac{C_2}{2}R , \end{aligned}
\end{eqnarray*}}%
and this completes the proof of \eqref{ape1}. 
\end{proof}


\begin{remark}
\label{Rmk:apemu=0}
If $\mu=0$, then  \eqref{onape13} and the condition
\(F^\epsi_0 \geq 0\) yield
{\setlength\arraycolsep{0.5pt}
\begin{eqnarray*}
&&
\int_{\Tt^d}  m_0^\epsi g_1(x,m^\epsi_0, h(\bfm^\epsi_0))\,\dx \leq
C, 
\end{eqnarray*}}%
where   \(C\) is a positive constant that depends only on the problem
data.
Moreover, using   \eqref{G:g ub} with $\delta =
\frac{C_2}{2\fc_2}$ in \eqref{onape12} first, and
then \eqref{G:g lb} together with the condition
\(F^\epsi_0 \geq 0\), it follows that
{\setlength\arraycolsep{0.5pt}
\begin{eqnarray*}
&&
\int_{\Tt^d} - g_1(x,m^\epsi_0, h(\bfm^\epsi_0))\,\dx \leq
C, 
\end{eqnarray*}}%
where   \(C\) is a positive constant that depends
only on the problem
data. These two estimates  are uniform in \(\epsi\) and play a major role 
in the study of the limit of \eqref{RP}
as $\epsi_1\to 0^+$ and $\epsi_2\to0^+$.
\end{remark}



In what follows,  if $\epsi\in (0,1)^2$ and
$\mu \in \Lambda^\epsi
$, where  $\Lambda^\epsi$ is
the set defined
in \eqref{Lambda}, then $(m^\epsi_\mu,u^\epsi_\mu)$ represents
an arbitrary solution
to \eqref{Pe}, which belongs to \(D(A^\epsi_\mu)  \cap{C^\infty}(\Tt^d)
\times  {C^\infty}(\Tt^d)\).
The next lemma provides a uniform bound with
respect to \(\mu \in [0,1]\) on the infimum
of such solutions.

\begin{lem}\label{Lem:m>0 mu}
Assume that Assumptions \eqref{g} and \eqref{bddH} hold. Let $\epsi\in (0,1)^2$ be such that $\Lambda^\epsi \not= \emptyset$, where  $\Lambda^\epsi$ is the set defined
in \eqref{Lambda}. Then, 
\begin{eqnarray}\label{m>0 mu}
\begin{aligned}
\inf_{ \mu \in \Lambda^\epsi}\inf_{\Tt^d} m^\epsi_\mu >0.
\end{aligned}
\end{eqnarray}
\end{lem}

\begin{proof}

We begin by proving that
\begin{eqnarray}\label{Lip(m) mu}
\begin{aligned}
L^{\epsi_1}:=\sup_{\mu \in \Lambda^\epsi} \Vert m^\epsi_\mu
\Vert_{W^{1,\infty}(\Tt^d)}<+\infty.
\end{aligned}
\end{eqnarray}
Because the integral terms in \eqref{ape1}
are non-negative, we deduce that
\begin{eqnarray*}
\begin{aligned}
\sup_{\mu \in \Lambda^\epsi} \bigg[\epsi_1 \int_{\Tt^d}
\big((u^\epsi_\mu)^2 + (\Delta^p u^\epsi_\mu)^2
+ (m^\epsi_\mu)^2
+ (\Delta^p m^\epsi_\mu)^2 
\big) \,\dx \bigg] \leq C
\end{aligned}
\end{eqnarray*}
for some positive constant, $C$,
that depends only on the problem
data.
From this estimate and using integration by
parts, we conclude that
\begin{eqnarray}
\begin{aligned}\label{W2p}
\sup_{\mu \in \Lambda^\epsi}
 \big[ \Vert u^\epsi_\mu\Vert_{W^{2p,2}(\Tt^d)}
+ \Vert m^\epsi_\mu\Vert_{W^{2p,2}(\Tt^d)}
\big] \leq C
 \end{aligned}
\end{eqnarray}
for some positive constant, $C$,
that depends only on the problem
data and on \(\epsi\).
Hence, using \eqref{SobEmb},
$u^\epsi_\mu$, $m^\epsi_\mu\in C^{4,\lambda}(\Tt^d)$ with 
\begin{eqnarray}\label{W4infty}
\begin{aligned}
\sup_{\mu \in \Lambda^\epsi}
 \big[ \Vert u^\epsi_\mu\Vert_{W^{4,\infty}(\Tt^d)}
+ \Vert m^\epsi_\mu\Vert_{W^{4,\infty}(\Tt^d)}
\big] \leq C
\end{aligned}
\end{eqnarray}
for some positive constant, $C$,
that depends only on the problem
data and on \(\epsi\).
In particular, \eqref{Lip(m) mu} holds.

We now establish \eqref{m>0 mu}. By contradiction, assume that $\inf_{\mu \in
\Lambda^\epsi}\inf_{\Tt^d} m^\epsi_\mu=0$.
Set $\bar m^\epsi_\mu := \inf_{\Tt^d} m^\epsi_\mu
$ and let 
$(\mu_n)_{n\in\Nn}\subset
\Lambda^\epsi$ be such that
\begin{eqnarray*}
\begin{aligned}
\lim_{n\to\infty} \bar m^\epsi_{\mu_n} =0.
\end{aligned}
\end{eqnarray*}
Note that $\bar m^\epsi_{\mu_n} > 0$ by the definition
of $D(A^\epsi_{\mu_n})$ (see
\eqref{domAe}).
Because $m^\epsi_{\mu_n}$ is continuous in $\Tt^d$,
there exists $\bar x^\epsi_{\mu_n} \in \Tt^d$ satisfying
 $m^\epsi_{\mu_n}(\bar x^\epsi_{\mu_n}) = \bar m^\epsi_{\mu_n}$. Let $Q(\bar x^\epsi_{\mu_n},
\bar
m^\epsi_{\mu_n})$ be
the cube centered at \(\bar x^\epsi_{\mu_n} \)
and with side length \(m^\epsi_{\mu_n}\).
For all  sufficiently large $n\in\Nn$, $0<\bar
m^\epsi_{\mu_n}\leq \min\{ \frac{\epsi_1}{4},
\frac{\epsi_1}{2L^{\epsi_1}}\}$. Therefore, for all such $n\in\Nn$ and for a.e. 
 $x\in Q(\bar x^\epsi_{\mu_n},
\bar
m^\epsi_{\mu_n})$,
we have
\begin{eqnarray*}
\begin{aligned}
0< m^\epsi_{\mu_n} (x) \leq \frac{\bar
m^\epsi_{\mu_n}
L^{\epsi_1}}{2} +
\bar m^\epsi_{\mu_n} \leq \frac{\epsi_1}{2}.
\end{aligned}
\end{eqnarray*}
Consequently, using  \eqref{ape1},
the non-negativeness of $-\beta_{\epsi_1}(m^\epsi_{\mu_n}
)$,   and the identity
$\beta_{\epsi_1}(s) = -s^{-q}$ if $0<s \leq \frac{\epsi_1}{2}$,
we obtain
\begin{eqnarray*}
\begin{aligned}
C
&\geq \int_{\Tt^d} -\beta_{\epsi_1}(m^\epsi_{\mu_n}
)\,\dx
\geq \int_{Q(\bar x^\epsi_{\mu_n},\bar m^\epsi_{\mu_n})}
\frac{1}{(m^\epsi_{\mu_n})^q}
\,\dx \\
&\geq   \frac{ (\bar
m^\epsi_{\mu_n})^d}{
\Big(\frac{\bar
m^\epsi_{\mu_n} L^{\epsi_1}}{2}
+ \bar
m^\epsi_{\mu_n}\Big)
^q} =\Big(\frac{ 2}{L^{\epsi_1}+2}\Big)^q
\frac{1}{
(\bar m^\epsi_{\mu_n})^{q-d}},
\end{aligned}
\end{eqnarray*}
where   \(C\) is a positive constant that depends
only on the problem
data. Because  $q>d$ and $\lim_{n\to\infty} \bar m^\epsi_{\mu_n} =0$, we have a contradiction when
we let $n\to\infty$ in the estimate above. 
\end{proof}


\begin{remark}\label{Rmk:bds m}
Under the conditions of the previous lemma,  the
estimates in \eqref{m>0 mu} and \eqref{W4infty} imply that
\begin{eqnarray*}
\begin{aligned}
\underbar{m}^\epsi:= \inf_{\mu \in \Lambda^\epsi}\inf_{\Tt^d} m^\epsi_\mu \quad\hbox{and}\quad {\overline{m}}^\epsi:= \sup_{\mu \in \Lambda^\epsi}\sup_{\Tt^d}
m^\epsi_\mu
\end{aligned}
\end{eqnarray*}
satisfy
\begin{eqnarray}\label{infsupm}
\begin{aligned}
0<\underline m^\epsi \leq \overline m^\epsi \leq C,
\end{aligned}
\end{eqnarray}
where   \(C\) is a positive constant that depends
only on the problem
data and on \(\epsi\). 
\end{remark}

Next, we prove uniform bounds in \(\mu\in
[0,1]\) for classical solutions to \eqref{Pe} with respect to the norm \(\Vert \cdot \Vert_{W^{k,\infty}(\Tt^{d})}\).
These bounds rely on the previous two lemmas and  play an important role in Subsection~\ref{Subs:L=[0,1]}, where we establish the existence of classical solutions to  \eqref{Pe}. 


\begin{lem}\label{Lem:Wkinfty}
Suppose that Assumptions \eqref{h},  \eqref{g},  \eqref{bddH}, and \eqref{H:smooth} are satisfied. Let $\epsi\in (0,1)^2$ be such that $\Lambda^\epsi
\not= \emptyset$, where  $\Lambda^\epsi$ is
the set defined
in \eqref{Lambda}. Then, for all $k\in\Nn_0$, there exists a positive constant, $C$, that depends
only on the problem
data and on \(\epsi\),
such that
\begin{eqnarray*}
\begin{aligned}
\sup_{\mu \in \Lambda^\epsi} \big(\Vert m^\epsi_\mu
\Vert_{W^{k,\infty}(\Tt^d)} + \Vert u^\epsi_\mu
\Vert_{W^{k,\infty}(\Tt^d)}\big) \leq C.
\end{aligned}
\end{eqnarray*}
\end{lem}

\begin{proof}
To simplify the notation, we  drop the
dependence on $\epsi$ and $\mu$ of $\bfm^\epsi_\mu$, $m^\epsi_\mu,$
and $u^\epsi_\mu$. Accordingly, we 
write $\bfm$, $m$, and $u$ in place of $\bfm^\epsi_\mu$, $m^\epsi_\mu$,
and $u^\epsi_\mu$, respectively.
Moreover,  \(C\) is a positive constant  that depends
only on the problem
data and on \(\epsi\) and whose value may change from one expression to another.

Arguing as in the proof of Lemma~\ref{Lem:m>0 mu}
and using the hypothesis, we conclude that $m$, $u\in W^{4,\infty}(\Tt^d) \cap W^{2p,2}(\Tt^d)$,
that  \eqref{W2p} and \eqref{W4infty} hold, and that 
\begin{eqnarray}
&&\displaystyle  m +\Delta^{2p}m =f_1(m,u),
\label{eqnme1}\\
&&\displaystyle   u
+\Delta^{2p}u  = f_2(m,u), \label{eqnue1}
\end{eqnarray}
where
\begin{eqnarray}
&&\displaystyle f_1(m,u):= \frac{1}{\epsi_1}\Big[u + 
H_\mu (x,Du,D^2u,m,h(\bfm)) - \beta_{\epsi_1}(m) -\epsi_2(m -\Delta m)\Big],
\label{f1mu}\\
&&\displaystyle f_2(m,u):=\frac{1}{\epsi_1}\Big[- m + \div(m D_pH_\mu (x,Du,D^2u,m,h(\bfm))) \nonumber\\ 
&&\hskip20mm-(m
D_{M_{ij}} H_\mu (x,Du,D^2u,m,h(\bfm)))_{x_i x_j} +1 -\epsi_2(u- \Delta u)\Big]. \label{f2mu}
\end{eqnarray}

By \eqref{on h(m)}, \eqref{SobEmb}, and \eqref{W2p}, we also have that
{\setlength\arraycolsep{0.5pt}
\begin{eqnarray}
\begin{aligned}
\Vert h(\bfm)\Vert_{L^\infty(\Tt^d)} \leq C \Vert h(\bfm)\Vert_{W^{2p-4,2}(\Tt^d)} \leq C(1+ \Vert m\Vert_{W^{2p-4,2}(\Tt^d)} )\leq C. \label{h(m)bdd}
\end{aligned}
\end{eqnarray}}%
Hence, when
{\setlength\arraycolsep{0.5pt}%
\begin{eqnarray*}
\begin{aligned}
K:=\{(x,Du(x), D^2u(x), m(x), h(\bfm)(x))\!:\, x\in \Tt^d \}
\end{aligned}
\end{eqnarray*}}%
is thus defined, the estimates  \eqref{W4infty}, \eqref{infsupm}, and \eqref{h(m)bdd} yield that $\overline K$ is a compact subset of $\Tt^d \times \RR^d \times \RR^{d \times d} \times \EE \times \RR$ and there exists a ball, $B(0,C)$, centered
at the origin and of radius $C$ such that $\overline K \subset B(0,C)$. Using in addition Assumption \eqref{H:smooth}, it follows that,
for all $\kappa \in \NN_0$,
{\setlength\arraycolsep{0.5pt}
\begin{eqnarray}\label{Hmubdd}
\begin{aligned}
\Vert H_\mu\Vert_{W^{\kappa,\infty}(\overline K)} \leq \Vert H \Vert_{W^{\kappa,\infty}(\overline K)} + 
\Vert \widetilde H\Vert_{W^{\kappa,\infty}(\overline K)} \leq C.
\end{aligned}
\end{eqnarray}}%

Next, we recall that if  $q>\frac{d}{2}$, then   $W^{q,2}(\Tt^d)$ is an algebra. Thus, because
 $p$ is such that
$2p-4>\frac{d}{2} +1$, from Remark~\ref{Rmk:bds m}, \eqref{W2p}, \eqref{on h(m)}, and \eqref{Hmubdd}, it follows that 
\begin{eqnarray*}
&&f_1(m,u)\in W^{2,2}(\Tt^d),\quad \Vert f_1(m,u) \Vert_{W^{2,2}(\Tt^d)}\leq c=c(C,\Vert 
\beta_{\epsi_1}\Vert_{W^{2,2}([\underline m^\epsi.\overline
m^\epsi])}), \\\
&& f_2(m,u)\in L^2(\Tt^d),\quad \Vert
f_2(m,u) \Vert_{L^{2}(\Tt^d)}\leq \bar c= \bar c(C).
\end{eqnarray*}
Therefore, \eqref{eqnme1}, \eqref{eqnue1}, and the elliptic regularity theory yield
\begin{eqnarray*}
&&m \in W^{4p+2,2}(\Tt^d) \enspace \hbox{with} \enspace \Vert m \Vert_{W^{4p+2,2}(\Tt^d)}\leq c_1=c_1(c),\\
&& u \in W^{4p,2}(\Tt^d) \enspace \hbox{with}
\enspace \Vert
u \Vert_{W^{4p,2}(\Tt^d)}\leq \bar c_1=\bar
c_1(\bar
c).
\end{eqnarray*}
Then, going back to \eqref{f1mu} and \eqref{f2mu}, and arguing as above, we have
\begin{eqnarray*}
&&f_1(m,u)\in W^{4p-2,2}(\Tt^d),\quad \Vert
f_1(m,u) \Vert_{W^{4p-2,2}(\Tt^d)}\leq c_2=c_2(c_1, \bar c_1,\Vert
\beta_{\epsi_1}\Vert_{W^{4p-2,2}([\underline
m^\epsi.\overline
m^\epsi])}), \\\
&& f_2(m,u)\in W^{4p-4,2}(\Tt^d),\quad \Vert
f_2(m,u) \Vert_{W^{4p-4,2}(\Tt^d)}\leq \bar
c_2= \bar c_2(c_1, \bar c_1).
\end{eqnarray*}
Arguing as before,
we conclude that 
\begin{eqnarray*}
&&m \in W^{8p-2,2}(\Tt^d) \enspace \hbox{with}
\enspace \Vert
m \Vert_{W^{8p-2,2}(\Tt^d)}\leq c_3=c_3(c_2,\bar
c_2),\\
&& u \in W^{8p-4,2}(\Tt^d) \enspace \hbox{with}
\enspace \Vert
u \Vert_{W^{8p-4,2}(\Tt^d)}\leq \bar
c_3=\bar
c_3(c_2,\bar
c_2).
\end{eqnarray*}
The conclusion follows
by iterating this ``bootstrap" argument and using the Sobolev embedding theorem. 
\end{proof}


\subsection{Existence of classical solutions to the perturbed mean-field game}\label{Subs:L=[0,1]}

Here, we prove that $\Lambda^\epsi=[0,1]$
for all $\epsi\in (0,1)^2$.
Consequently, by the definition of 
$\Lambda^\epsi$ in \eqref{Lambda}, 
we have that \eqref{RP} has 
a solution, $(u^\epsi, m^\epsi)$, with $m^\epsi>0$. 
To prove that $\Lambda^\epsi=[0,1]$, 
we show that
$\Lambda^\epsi$ is a nonempty set, relatively closed
and open in $[0,1]$. These properties are established
in the next three lemmas. 

\begin{lem}     
        Let $\Lambda^\epsi$ be the set defined in \eqref{Lambda}. Then, $\Lambda^\epsi \not= \emptyset$.
\end{lem}
\begin{proof}
To prove the lemma, we show that $1\in \Lambda^\epsi$.
We set
\[
f(c):=-c-\frac{1}{\gamma(1-c(\epsi_1+ \epsi_2))^\tau}  + (\epsi_1+\epsi_2)(1-c(\epsi_1+ \epsi_2)) + \beta_{\epsi_1}(1-c(\epsi_1+ \epsi_2)).
\]
Recall that
$\gamma\geq1>0$,  $\tau\geq0$, $\beta_{\epsi_1}(s) = 0$ if $s\geq \epsi_1,$ and $\lim_{s\to0^+} \beta_{\epsi_1}(s)=-\infty$.
Consequently, we have that 
$\lim_{c\to-\infty} f(c) = +\infty$ and  $\lim_{c\to(\frac{1}{\epsi_1+ \epsi_2})^-} f(c) = -\infty$. 
Therefore, there exists 
 $\bar c\in (-\infty,\frac{1}{\epsi_1+\epsi_2})$
 such that $f(\bar c)=0$. 

Set $m^\epsi_1:= 1-\bar c(\epsi_1+ \epsi_2)$ and $u^\epsi_1:= \bar c$. Then, $(m^\epsi_1,u^\epsi_1)\in
D(A^\epsi_1)  \cap C^\infty(\Tt^d)\times C^\infty(\Tt^d)$  and 
\begin{eqnarray*}
\begin{aligned}
A^\epsi_1\begin{bmatrix}m^\epsi_1 \\
u^\epsi_1 \\
\end{bmatrix} = \begin{bmatrix} -\bar c- 
\frac{1}{\gamma(1-\bar c(\epsi_1+ \epsi_2))^\tau}
 + (\epsi_1+\epsi_2)(1-\bar c(\epsi_1+ \epsi_2))
+ \beta_{\epsi_1}(1-\bar c(\epsi_1+ \epsi_2)) \\
1-\bar c(\epsi_1+ \epsi_2) +\epsi_1 \bar c +\epsi_2 \bar c -1  \\
\end{bmatrix} = \begin{bmatrix}0 \\
0  \\
\end{bmatrix}.
\end{aligned}
\end{eqnarray*}
This completes the proof. 
\end{proof}

\begin{lem}
Suppose that Assumptions \eqref{h}, \eqref{g},  \eqref{bddH}, and \eqref{H:smooth} hold.
 Let $\Lambda^\epsi$ be the set defined
in \eqref{Lambda}. Then, $\Lambda^\epsi$ is a closed subset of $[0,1]$.
\end{lem}

\begin{proof}
Let   $\{\mu_n\}_{n\in \Nn} \subset \Lambda^\epsi$ be a sequence converging to some $\mu_0\in [0,1]$.
We claim that   $\mu_0 \in \Lambda^\epsi$. To
see this, let $\{(m^\epsi_{\mu_n}, 
u^\epsi_{\mu_n})\}_{n\in\Nn} \subset
 D(A^\epsi_{\mu_n})  \cap{C^\infty}(\Tt^d) \times
 {C^\infty}(\Tt^d)$ be such that 
\begin{eqnarray}\label{Aepsimun}
\begin{aligned}
A^\epsi_{\mu_n}\begin{bmatrix}m^\epsi_{\mu_n} \\
u^\epsi_{\mu_n} \\
\end{bmatrix} = \begin{bmatrix}0 \\
0 \\
\end{bmatrix}.
\end{aligned}
\end{eqnarray}
By Lemma~\ref{Lem:Wkinfty}, 
\begin{eqnarray*}
\begin{aligned}
\sup_{n\in\Nn} \big(\Vert m^\epsi_{\mu_n}
\Vert_{W^{4p+1,2}(\Tt^d)} + \Vert u^\epsi_{\mu_n}
\Vert_{W^{4p+1,2}(\Tt^d)}\big) \leq C
\end{aligned}
\end{eqnarray*}
for some positive
constant, $C$, that depends only on the problem
data and on \(\epsi\). Hence, by the Rellich-Kondrachov theorem, we can find  $(m^\epsi,u^\epsi)\in W^{4p+1,2}(\Tt^d)
\times W^{4p+1,2}(\Tt^d)$ such that, up to a (not relabeled) subsequence, \((m^\epsi_{\mu_n} , u^\epsi_{\mu_n} )_{n \in \NN} \) converges to \((m^\epsi, u^\epsi)\) weakly in \(W^{4p+1,2} (\Tt^d)\) and strongly in \(W^{4p,2}
(\Tt^d)\) and   
\begin{eqnarray*}
\begin{aligned}
D^\alpha m^\epsi_{\mu_n} (x) \to D^\alpha m^\epsi (x) \enspace \hbox{and}\enspace D^\alpha u^\epsi_{\mu_n} (x) \to D^\alpha
u^\epsi (x) 
\end{aligned}
\end{eqnarray*}
as $n\to\infty$ for \aev\ $x\in\Tt^d$ and for
all  $\alpha\in\Nn^d_0$ with $|\alpha| \leq
4p$. Using \eqref{hFD} with \(\bfm\) and $\bfm_0$ replaced by \(\bfm^\epsi_{\mu_n}\) and $\bfm^\epsi$, respectively, up to a further (not relabeled) subsequence, we also have that for \aev\ $x\in\Tt^d$
and for all  $\alpha\in\Nn^d_0$ with $|\alpha|
\leq 4p$,
\begin{eqnarray*}
\begin{aligned}
D^\alpha h(\bfm^\epsi_{\mu_n} )(x) \to D^\alpha
h(\bfm^\epsi )(x)   
\end{aligned}
\end{eqnarray*}  as $n\to\infty$. Moreover, using  Lemma~\ref{Lem:Wkinfty} once more and in view of Lemma~\ref{Lem:m>0 mu}, $(m^\epsi,u^\epsi)\in W^{4p+1,2}(\Tt^d)
\times W^{4p+1,2}(\Tt^d) \cap C^\infty(\Tt^d)
\times {C^\infty}(\Tt^d)$ with $\inf_{\Tt^d} m^\epsi>0$. By Assumption~\eqref{H:smooth} and because $\beta_{\epsi_1}(\cdot)$ is smooth, passing \eqref{Aepsimun} to the limit as $n\to\infty$, we conclude that 
\begin{eqnarray*}
\begin{aligned}
A^\epsi_{\mu_0}\begin{bmatrix}m^\epsi
\\
u^\epsi \\
\end{bmatrix} = \begin{bmatrix}0 \\
0 \\
\end{bmatrix}.
\end{aligned}
\end{eqnarray*}
Hence, $\mu_0\in \Lambda^\epsi$. 
\end{proof}

\begin{lem}
Under Assumptions \eqref{h}, \eqref{g}, and \eqref{bddH}--\eqref{H:mon}, the set
  $\Lambda^\epsi$  defined
in \eqref{Lambda} is an open subset of $[0,1]$.
\end{lem}

\begin{proof}
Fix $\mu_0\in \Lambda^\epsi$. We want to show that there exists a neighborhood of $\mu_0$ in $[0,1]$ contained in $\Lambda^\epsi$. That will be a consequence of the implicit function theorem in Banach spaces (see, for example, \cite{Die}).

Fix a solution, $(m_0, u_0) \in
 D(A^\epsi_\mu)  \cap{C^\infty}(\Tt^d)
\times
 C^\infty(\Tt^d)$,   to \eqref{Pe}
with $\mu=\mu_0$. 

{ In this proof, to simplify
        the notation, we omit the
        dependence on \(\epsi\) and \(\mu_0\). }

To use the implicit
function theorem,
we have to prove that the Fr\'echet derivative, $\Ll: W^{4p,2}(\Tt^d) \times  W^{4p,2}(\Tt^d) \to L^2(\Tt^d) \times L^2(\Tt^d)$, of 
\begin{eqnarray*}
\begin{aligned}
\begin{bmatrix}m
\\
u \\
\end{bmatrix} \mapsto A^\epsi_{\mu_0} \begin{bmatrix}m
\\
u \\
\end{bmatrix}
\end{aligned}
\end{eqnarray*}
at $(m_0, u_0)$ is an isomorphism from  $W^{4p,2}(\Tt^d) \times  W^{4p,2}(\Tt^d)$ to $L^2(\Tt^d) \times L^2(\Tt^d)$.

Given $(\eta,v)\in W^{4p,2}(\Tt^d) \times
 W^{4p,2}(\Tt^d)$, we have that
\begin{eqnarray*}
\begin{aligned}
\Ll \begin{bmatrix}\eta
\\
v \\
\end{bmatrix} = 
\begin{bmatrix}
-v - \ell_1(\eta, v)+\ell_2(\eta ) \\
\eta- \div(\eta D_{p} H_{\mu_0} + m_0  \ell_3(\eta, v)) + (\eta D_{M_{ij}} H_{\mu_0}+m_0 \ell_4(\eta, v))_{x_i x_j} +\ell_5( v) \\
\end{bmatrix},
\end{aligned}
\end{eqnarray*}
where, recalling that \(\mathfrak{H}_{m_0}  \) is the Fr\'echet derivative of 
\(m\in  W^{4p,2}(\Tt^d) \mapsto h(\bfm)
\in  W^{4p,2}(\Tt^d)\) at \(m_0\) (see \eqref{hFD}), 
{\setlength\arraycolsep{0.5pt}%
\begin{eqnarray*}
&&\ell_1(\eta, v):= D_p H_{\mu_0} \cdot Dv + D_M H_{\mu_0} : D^2v + D_m H_{\mu_0} \eta + D_\theta
H_{\mu_0} \mathfrak{H}_{m_0}(\eta),\\
&& \ell_2(\eta ):= \epsi_1 (\eta + \Delta^{2p}\eta) + \epsi_2 (\eta - \Delta \eta)+ \beta'_\epsi (m_0)\eta, \\
&&\ell_3(\eta, v):= D_{pp}^2 H_{\mu_0}  Dv +
D_{M_{ij}p}^2 H_{\mu_0} D^2_{ij}v + D_{mp}^2 H_{\mu_0} \eta + D_{\theta p}^2
H_{\mu_0} \mathfrak{H}_{m_0}(\eta),\\
&& \ell_4(\eta, v):= D^2_{pM_{ij}} H_{\mu_0} \cdot Dv +
D_{MM_{ij}}^2 H_{\mu_0} : D^2v + D_{m M_{ij}}^2
H_{\mu_0} \eta + D_{\theta M_{ij}}
H_{\mu_0} \mathfrak{H}_{m_0}(\eta), \\
&& \ell_5( v):= \epsi_1 (v + \Delta^{2p} v)
+ \epsi_2 (v - \Delta v),
\end{eqnarray*}}%
and the argument of  $H_{\mu_0}$ and all its partial derivatives is $(x, D u_0,  D^2 u_0, m_0, 
 h(\bfm_0))$.

Let
{\setlength\arraycolsep{0.5pt}%
\begin{eqnarray*}
\begin{aligned}
K_0:=\{(x,Du_0(x), D^2u_0(x), m_0(x), h(\bfm_0)(x))\!:\,
x\in \Tt^d \}. 
\end{aligned}
\end{eqnarray*}}%
Arguing as in Lemma~\ref{Lem:Wkinfty}, for all $\kappa \in \NN_0$, we have
{\setlength\arraycolsep{0.5pt}
\begin{eqnarray}\label{Hmu0bdd}
\begin{aligned}
\Vert H_\mu\Vert_{W^{\kappa,\infty}(\overline
K_0)} \leq \Vert H \Vert_{W^{\kappa,\infty}(\overline
K_0)} + 
\Vert \widetilde H\Vert_{W^{\kappa,\infty}(\overline
K_0)} \leq C
\end{aligned}
\end{eqnarray}}%
for some positive
constant, $C$, that depends only on the problem
data and on \(\epsi\).

Note also that $\Ll(\eta,v) \in L^2(\Tt^d) \times L^2(\Tt^d)$ because $W^{4p-4,2}(\Tt^d)$ is an algebra,    $\inf_{\Tt^d} m_0>0$, and in view of \eqref{on h(m)} and \eqref{Hmu0bdd}. 
We will use the Lax-Milgram theorem
to prove that \(\mathcal{L}\) defines an isomorphism.

Fix $\bar z = (\bar z_1,\bar z_2)  \in L^2(\Tt^d)
\times L^2(\Tt^d)$.  Next, we
show that there exists a unique $\bar w=(\bar \eta,
\bar v) \in W^{4p,2}(\Tt^d) \times  W^{4p,2}(\Tt^d)$ such that $\Ll \bar w = \bar z$. 

Define a map $L: W^{2p,2}(\Tt^d) \times  W^{2p,2}(\Tt^d)
 \to \RR$ by  $L[w] := ( \bar z , w)_{L^2(\Tt^d)\times
L^2(\Tt^d)}$ for $w=(\eta, v) \in W^{2p,2}(\Tt^d)
\times  W^{2p,2}(\Tt^d)$. Clearly, $L$ is  linear and bounded  in $W^{2p,2}(\Tt^d)
\times  W^{2p,2}(\Tt^d)$.

Now, let  $B: W^{2p,2}(\Tt^d) \times  W^{2p,2}(\Tt^d)  \to \RR$ be the bilinear form defined for $w_1=(\eta_1, v_1)$, $w_2=(\eta_2, v_2) \in W^{2p,2}(\Tt^d) \times  W^{2p,2}(\Tt^d)$ by
\begin{eqnarray*}
\begin{aligned}
B[w_1,w_2] := \int_{\Tt^d} \big[ &-v_1 \eta_2  + \eta_1 v_2 - \ell_1(\eta_1, v_1)\eta_2  +m_0D_{p} H_{\mu_0}
\cdot Dv_2+\eta_1 \ell_3(\eta_1, v_1)\cdot Dv_2 \\
&  +\eta _1D_M H_{\mu_0}: D^2v_2+m_0
\ell_4(\eta_1, v_1) D^2_{ij} v_2  \\
&  +\epsi_1 (\eta_1 \eta_2 + \Delta^p \eta_1 \Delta^p \eta_2 + v_1 v_2 + \Delta^p v_1 \Delta^p v_2)\big) + \beta_{\epsi_1}'(m_0)\eta_1
\eta_2  \\
&  +\epsi_2 (\eta_1 \eta_2 + \nabla \eta_2 \nabla \eta_2 + v_1 v_2 + \nabla v_1 \nabla
v_2)\big] \,\dx.
\end{aligned}
\end{eqnarray*}
We start by observing that if $w_1=(\eta_1, v_1) \in W^{4p,2}(\Tt^d)
\times  W^{4p,2}(\Tt^d)$ and $w_2=(\eta_2, v_2) \in W^{2p,2}(\Tt^d)
\times  W^{2p,2}(\Tt^d)$, then
\begin{eqnarray}\label{BandL}
\begin{aligned}
B[w_1,w_2] = (\Ll w_1, w_2)_{L^2(\Tt^d)\times L^2(\Tt^d)}.
\end{aligned}
\end{eqnarray}
Moreover, for $w=(\eta,v)
\in W^{2p,2}(\Tt^d)
\times  W^{2p,2}(\Tt^d)$,
\begin{eqnarray*}
\begin{aligned}
B[w,w] &= B_1[w,w]+ \int_{\Tt^d}   \epsi_1(\eta^2 + (\Delta^{p} \eta)^2 + v^2 + (\Delta^{p} v)^2)
+ \beta_{\epsi_1}'(m_0)\eta^2\,\dx \\
&\qquad + \int_{\Tt^d}  \epsi_1(\eta^2
+ |\nabla\eta|^2 + v^2
+ |\nabla v|^2)\, \dx,
\end{aligned}
\end{eqnarray*}
where
{\setlength\arraycolsep{0.5pt}
\begin{eqnarray*}
 B_1[w,w]:= &&\int_{\Tt^d} \big[ - \ell_1(\eta, v)\eta 
+m_0D_{p} H_{\mu_0}
\cdot Dv+\eta \ell_3(\eta, v)\cdot
Dv   +\eta D_M H_{\mu_0}: D^2 v\\ 
&&\qquad + m_0 \ell_4(\eta, v) D^2_{ij} v\big]\, \dx.
\end{eqnarray*}}%
We claim that  
{\setlength\arraycolsep{0.5pt}
\begin{eqnarray}
&& B_1[w,w] \geq 0. \label{B1pos} 
\end{eqnarray}}%
Assume that \eqref{B1pos} holds. Then, using integration
by parts, Gagliardo-Nirenberg's interpolation inequalities,
and the fact that $m_0>0$ and $\beta_{\epsi_1}'(m_0)\geq
0$ in $\Tt^d$, it follows that there exists
a positive constant, $C=C(\epsi, p, d, \Tt^d)$,
such that
{\setlength\arraycolsep{0.5pt}
\begin{eqnarray*}
&& B[w,w] \geq C \Vert w\Vert^2_{W^{2p,2}(\Tt^d)
\times  W^{2p,2}(\Tt^d)}.
\end{eqnarray*}}%
Hence, $B[\cdot,\cdot]$ is coercive in $W^{2p,2}(\Tt^d) \times  W^{2p,2}(\Tt^d)$. 

Next, we prove that \eqref{B1pos} is a consequence of the monotonicity
of \(H\) and \(\widetilde
H\) encoded in Assumption \eqref{H:mon}. In fact, let \(\widetilde A\) and \(A_{\mu_0}\) be the operators associated with \(\widetilde H\) and \(H_{\mu_0}\), respectively (see \eqref{A}). Then, \(A_{\mu_0} = (1-\mu_0) A + \mu_0 \widetilde A\). Since both  \( A\) and \(\widetilde A\)
satisfy Assumption \eqref{H:mon}, the same holds for \(A_{\mu_0}\). Consequently,
{\setlength\arraycolsep{0.5pt}%
\begin{eqnarray*}
\begin{aligned}
0 \leq \lim_{\delta \to 0^+} \frac{1}{\delta^2} \bigg( A^\epsi_{\mu_0} \begin{bmatrix} m_0 + \delta \eta
\\
u_0 + \delta v \\
\end{bmatrix}
-  A^\epsi_{\mu_0} \begin{bmatrix} m_0

\\
u_0  \\
\end{bmatrix}, \begin{bmatrix} \delta \eta
\\
 \delta v \\
\end{bmatrix} \bigg)_{L^2(\Tt^d)\times
L^2(\Tt^d)} = B_1[w,w],
\end{aligned}
\end{eqnarray*}}%
which proves \eqref{B1pos}.

We also have that $B[\cdot,\cdot]$   is bounded in $W^{2p,2}(\Tt^d)
\times  W^{2p,2}(\Tt^d)$; more precisely,
\begin{eqnarray*}
\begin{aligned}
|B[w_1, w_2]| \leq C \Vert w_1\Vert_{W^{2p,2}(\Tt^d)
\times  W^{2p,2}(\Tt^d)} \Vert w_2 \Vert_{W^{2p,2}(\Tt^d)
\times  W^{2p,2}(\Tt^d)}
\end{aligned}
\end{eqnarray*}
for some positive
constant, $C$, that depends only on the problem
data and on \(\epsi\)
 in view of Remark~\ref{Rmk:bds m}, Lemma~\ref{Lem:Wkinfty}, \eqref{on h(m)}, and \eqref{Hmu0bdd}.

By the Lax-Milgram theorem, there exists a unique $\bar w = (\bar \eta, \bar v)\in W^{2p,2}(\Tt^d) \times  W^{2p,2}(\Tt^d)$ such that
\begin{eqnarray}\label{LaxMil}
\begin{aligned}
B[\bar w, w] = L[w] =( \bar z , w)_{L^2(\Tt^d)\times
L^2(\Tt^d)}  
\end{aligned}
\end{eqnarray}
for all $w \in W^{2p,2}(\Tt^d)
\times  W^{2p,2}(\Tt^d)$. From \eqref{BandL}, the proof is complete once we show that  $\bar w \in W^{4p,2}(\Tt^d)
\times  W^{4p,2}(\Tt^d)$. 
   
Equality \eqref{LaxMil} yields that $\bar \theta \in W^{2p,2}(\Tt^d)$ and $\bar
v \in W^{2p,2}(\Tt^d)$ are weak solutions of 
\begin{eqnarray*}
\begin{aligned}
\epsi (\bar \eta + \Delta^{2p} \bar \eta ) = \bar v +\ell_1(\bar \eta,\bar v)-\ell_2(\bar\eta ) + \bar z_1  =: g_1[\bar\eta, \bar v]
\end{aligned}
\end{eqnarray*}
and
\begin{eqnarray*}
\begin{aligned}
\epsi (\bar v + \Delta^{2p} \bar v
) =& - \bar\eta + \div(m_0 D_{p} H_{\mu_0} +\bar \eta \ell_3(\bar\eta,\bar
v)) \\
&-(\bar\eta D_{M_{ij}} H_{\mu_0}+m_0 \ell_4(\bar \eta, \bar
v))_{x_i x_j} -\ell_5(\bar v) + \bar z_2  =: g_2[\bar\eta,
\bar v],
\end{aligned}
\end{eqnarray*}
respectively, where the argument of  $H_{\mu_0}$ and all its
partial derivatives is $(x, D u_0,  D^2
u_0, m_0, 
 h(\bfm_0))$. Arguing as above, we have that $g_1[\bar\eta,
\bar v]$, $g_2[\bar\eta,
\bar v] \in L^2(\Tt^d)$; hence, by  the elliptic regularity theory, it follows that $\bar\eta \in W^{4p,2}(\Tt^d)$ and $\bar
v \in W^{4p,2}(\Tt^d)$.
\end{proof}

As a result of the previous lemmas, we obtain \eqref{Lambda=01}:

\begin{corollary}\label{Cor:L=[0,1]}
Assume that assumptions \eqref{h}, \eqref{g}, and \eqref{bddH}--\eqref{H:mon} are satisfied and let $\Lambda^\epsi$ be the set
defined in \eqref{Lambda}. Then, $\Lambda^\epsi
=[0,1]$.
\end{corollary}

The proof of Proposition \ref{RPSol} is now a simple matter. 

\begin{proof}[Proof of Proposition \ref{RPSol}]
By Corollary \ref{Cor:L=[0,1]}, $0\in \Lambda^\epsi$. Therefore, the 
claim in the statement holds. 
\end{proof}

\section{Existence of weak solutions}
\label{Sect:pmthm}

Our next goal is to establish the existence of weak solutions to \eqref{P}.
Because the parameter $\mu$ does not play any role in this section, we set
$\mu=0$ and $A^\epsi=A^\epsi_0$, and we recall that
\eqref{RP} corresponds to \eqref{Pe} with \(\mu=0\).
We proceed in three steps. First, we prove that the operator $A^\epsi$
is monotone. Next, we prove estimates for solutions of \eqref{RP} 
that are uniform in $\epsi$. Finally, combining the monotonicity and these estimates,  we use Minty's
method to obtain a weak solution to \eqref{P}.

\begin{lemma}\label{Lem:Ae mon}
Suppose that Assumption~\eqref{H:mon} holds.  Then, the operator $A^\epsi$ given by 
\eqref{domAe}--\eqref{Ae} (with $\mu=0$) is monotone in $L^2(\Tt^d)\times L^2(\Tt^d)$. More precisely, for all 
$(m_1, u_1)$, $(m_2, u_2)\in D(A^\epsi) \cap C^\infty(\Tt^d) \times C^\infty(\Tt^d)$, we have that 
{\setlength\arraycolsep{0.5pt}%
\begin{eqnarray*}
\begin{aligned}
\bigg(A^\epsi
\begin{bmatrix}m_1 \\
u_1 \\
\end{bmatrix} 
- 
A^\epsi
\begin{bmatrix}m_2 \\
u_2 \\
\end{bmatrix}, \begin{bmatrix}m_1 \\
u_1 \\
\end{bmatrix} 
-  
\begin{bmatrix}m_2 \\
u_2 \\
\end{bmatrix}  
\bigg)_{L^2(\Tt^d)\times L^2(\Tt^d)}
 \geq 0.
\end{aligned}
\end{eqnarray*}}%
\end{lemma}
\begin{proof}
        Let $(m_1, u_1)$, $(m_2, u_2)\in D(A^\epsi)
        \cap C^\infty(\Tt^d) \times C^\infty(\Tt^d)$. Using the definition of $A^\epsi$ and \(A\) 
(see \eqref{Ae} and \eqref{A}), we obtain
        {\setlength\arraycolsep{0.5pt}%
                \begin{eqnarray*}
                        \begin{aligned}
                                &\bigg(A^\epsi
                                \begin{bmatrix}m_1 \\
                                        u_1 \\
                                \end{bmatrix} 
                                - 
                                A^\epsi
                                \begin{bmatrix}m_2 \\
                                        u_2 \\
                                \end{bmatrix}, \begin{bmatrix}m_1 \\
                                u_1 \\
                        \end{bmatrix} 
                        -  
                        \begin{bmatrix}m_2 \\
                                u_2 \\
                        \end{bmatrix}  
                        \bigg)_{L^2(\Tt^d)\times L^2(\Tt^d)} \\
                        &\qquad = \bigg(A
                        \begin{bmatrix}m_1 \\
                                u_1 \\
                        \end{bmatrix} 
                        - 
                        A
                        \begin{bmatrix}m_2 \\
                                u_2 \\
                        \end{bmatrix}, \begin{bmatrix}m_1 \\
                        u_1 \\
                \end{bmatrix} 
                -  
                \begin{bmatrix}m_2 \\
                        u_2 \\
                \end{bmatrix}  
                \bigg)_{L^2(\Tt^d)\times L^2(\Tt^d)}\\
        &\qquad \quad +\int_{\Tt^d} (\beta_{\epsi_1} (m_1)
        - \beta_{\epsi_1} (m_2)) (m_1 - m_2)\,\dx\\
        &\qquad\quad + \epsi_1\int_{\Tt^d} \left[( m_1 - m_2)^2 + (\Delta^p
        m_1 - \Delta^p m_2)^2  +(u_1 - u_2)^2 + (\Delta^p u_1 - \Delta^p u_2)^2\right] \,\dx \\
        &\qquad\quad + \epsi_2\int_{\Tt^d} \left[( m_1 -
        m_2)^2 + (D
        m_1 - D m_2)^2  +(u_1 - u_2)^2 + (D
        u_1 - D u_2)^2\right] \,\dx.
\end{aligned}
\end{eqnarray*}}%
The conclusion follows because  $\beta_{\epsi_1}(\cdot)$ is a non-decreasing function and   \(A\) satisfies \eqref{H:mon}.
\end{proof}

Next, we give bounds for solutions of \eqref{RP}. 
\begin{pro}
\label{punifbd}
Suppose that Assumptions \eqref{g} and \eqref{bddH} hold.
Then, there exists a positive constant, $C$,
such that, for any $\epsi\in (0,1)^2$
and any solution  $(u^\epsi, m^\epsi)\in D(A^\epsi)
$  to 
\eqref{RP}, we have
\begin{equation}
\label{uW1g1}
\int_{\Tt^d}
|Du^\epsi|^\gamma \,\dx\leq C,
\end{equation}
\begin{equation}
\label{uW1g2}
\left|\int_{\Tt^d} u^\epsi\,\dx\right|\leq C,
\end{equation}
\begin{equation}
\label{Est3}
\int_{\Tt^d}|\beta_{\epsi_1}(m^\epsi)|\,\dx\leq C,
\end{equation}
\begin{equation}
\label{Est4}
\intt m^\epsi g_1(x, m^\epsi, h(\bfm^\epsi))\,\dx\leq C,
\end{equation}
and 
\begin{equation}
\label{bddsg1}
\int_{\Tt^d} - g_1(x,m^\epsi, h(\bfm^\epsi))\,\dx
\leq
C.
\end{equation}
\end{pro}
\begin{proof}To simplify the notation,  \(C\) represents a positive constant depending only on the problem data and whose value may change from one expression to another.

Estimates \eqref{Est4} and \eqref{bddsg1}
follow from Remark~\ref{Rmk:apemu=0}.

Next, we prove \eqref{uW1g1} and \eqref{Est3}. Because all integral
terms in (\ref{ape1}) are nonnegative, $0\leq \tau<1$, $m^\epsi>0$, and \(\beta_{\epsi_1}(\cdot)
\leq 0 \), it follows
from Lemma~\ref{Lem:ape} with \(\mu=0\) that  
\begin{equation*}
C \geq \intt -\beta_{\epsi_1} (m^\epsi)\,\dx 
=\int_{\Tt^d}|\beta_{\epsi_1}(m^\epsi)|\,\dx
\end{equation*}
and
{\setlength\arraycolsep{0.5pt}
        \begin{eqnarray*}
        C &&\geq 
        \int_{\Tt^d}
        |Du^\epsi|^\gamma (m^\epsi)^{1-\tau}\dx
        +  \int_{\Tt^d}
        |Du^\epsi|^\gamma (m^\epsi )^{-\tau}
        \, \dx  \nonumber \\
        && \geq 
        \int_{\{x\in \Tt^d\!: \, m^\epsi(x) \geq
                1\}}
        |Du^\epsi|^\gamma \dx
        +  \int_{\{x\in \Tt^d\!: \, m^\epsi(x)<
                1\}}
        |Du^\epsi|^\gamma 
        \, \dx  \nonumber \\
        &&=
        \int_{\Tt^d}
        |Du^\epsi|^\gamma \dx. 
        \end{eqnarray*}}%
Therefore, \eqref{uW1g1} and \eqref{Est3} hold.
 
Finally, we show \eqref{uW1g2}. 
Integrating the first equation in \eqref{RP}
over \(\Tt^d\), we obtain
{\setlength\arraycolsep{0.5pt}
        \begin{eqnarray}
        \int_{\Tt^d} u^\epsi \, \dx = &&-\int_{\Tt^d} H (x,Du^\epsi,D^2u^\epsi,m^\epsi,h(\bfm^\epsi)) \, \dx +  \int_{\Tt^d} \beta_{\epsi_1}(m^\epsi)\, \dx \nonumber\\
        &&+ (\epsi_1 + \epsi_2)\int_{\Tt^d}m^\epsi \, \dx. \label{meanu0}
        \end{eqnarray}}%
By  \eqref{ape1} with \(\mu=0\) and \eqref{Est3},
we have that
{\setlength\arraycolsep{0.5pt}
        \begin{eqnarray}
        &&
        \int_{\Tt^d} \big|\beta_{\epsi_1}(m^\epsi)\big| \,
        \dx + 
        (\epsi_1 + \epsi_2)\int_{\Tt^d}\big| m^\epsi  \big| \,\dx \nonumber\\
        &&\qquad\leq C + 
        2+(\epsi_1 + \epsi_2)\int_{\Tt^d} (m^\epsi)^2
        \, \dx  \leq C. \label{meanu1}
        \end{eqnarray}}%
Next, we estimate the upper and lower bounds of
\(\intt H\,\dx\) in \eqref{H:int b}. According to \eqref{G:sumg}, \eqref{bddsg1}, and \eqref{ape1} with \(\mu=0\), we have that
{\setlength\arraycolsep{0.5pt}
        \begin{eqnarray}
        \begin{aligned}
         &-C_4 \intt g(x,m^\epsi,
        h(\bfm^\epsi)) \,\dx +C_1   \bigg( 1+\int_{\Tt^d}   |Du^\epsi|^\gamma (m^\epsi)^{-\tau} \,\dx
        \bigg) 
        \leq C.\qquad \label{meanu2}
        \end{aligned}
        \end{eqnarray}}%
Similarly,
from \eqref{G:sumg}, \eqref{G:g+}, \eqref{G:g ub} with \(\delta = 1\), and \eqref{Est4}, we
get 
{\setlength\arraycolsep{0.5pt}
        \begin{eqnarray}
        \begin{aligned}
        & C_3\int_{\Tt^d} g(x,m^\epsi, h(\bfm^\epsi))\,\dx
- \frac{1}{C_1}\intt|Du^\epsi|^\gamma (m^\epsi)^{-\tau}\,\dx  + C_1 \\ 
 &\qquad\leq C_3(1+c)\max\bigg\{ \int_{\Tt^d} g_1(x,m,h(\bfm))
        \,\dx, \int_{\Tt^d} m\,\dx \bigg\}   
      +C_3c + C_1 \leq C.
        \qquad \label{meanu3}
        \end{aligned}
        \end{eqnarray}}%
Hence, in view of \eqref{H:int b}, \eqref{meanu2} and \eqref{meanu3} yield
{\setlength\arraycolsep{0.5pt}
        \begin{eqnarray}
        && \bigg|  -\int_{\Tt^d}
        H (x,Du^\epsi,D^2u^\epsi,m^\epsi,h(\bfm^\epsi))
        \, \dx \bigg| \leq C.\label{meanu4}
        \end{eqnarray}}%
Consequently, owing to \eqref{meanu0}, \eqref{meanu1}, and \eqref{meanu4}, we conclude that \eqref{uW1g2}
holds.    
\end{proof}

Next, we present the proof of Theorem~\ref{Thm:main}.

\begin{proof}[Proof of Theorem~\ref{Thm:main}]
As mentioned at the beginning of this section,
we set $A^\epsi= A^\epsi_0$. In view of Corollary~\ref{Cor:L=[0,1]} with $\mu=0$,
there exists $(m^\epsi, u^\epsi) \in
D(A^\epsi) \cap C^\infty (\Tt^d)\times C^\infty(\Tt^d)$
 such that 
\begin{eqnarray}\label{solmu=0}
\begin{aligned}
A^\epsi\begin{bmatrix}m^\epsi \\
u^\epsi \\
\end{bmatrix} = \begin{bmatrix}0 \\
0 \\
\end{bmatrix}
\end{aligned}
\end{eqnarray}
pointwise in $\Tt^d$. 

As usual, we assume that $\epsi_1$ and $\epsi_2$
take values on fixed sequences of positive
numbers converging to zero.
To simplify the notation, we
write \(\epsi \to 0^+\) instead of \( \epsi 
\to (0^+, 0^+).\)

We  now proceed in
two steps.

\uline{\textit{Step 1.}} In this step, we prove
that there exists \((\bfm,u) \in \mathcal{M}_{ac}(\Tt^d) \times
W^{1, \gamma} (\Tt^d)\) with $\int_{\Tt^d} m \,\dx =1$ such that, up to a
not relabeled subsequence, \((m^\epsi, u^\epsi) \weakly (m,u)\) weakly in \(  L^1(\Tt^d)\times
W^{1, \gamma} (\Tt^d)\)  as \( \epsi = (\epsi_1,
\epsi_2) 
\to  0^+.\)

 From \eqref{Est4} and Assumption \eqref{G:g
        ei}, it follows that there exists \(m \in L^1(\Tt^d)\) such that, up to a
not relabeled subsequence, \(m^\epsi
\weakly m\) weakly in \(  L^1(\Tt^d)\)  as \( \epsi
\to 0^+.\) Because $m^\epsi $ is nonnegative in $\Tt^d$, we
conclude that \(\bfm \in\mathcal{M}_{ac}(\Tt^d) \). 

Next, we observe that \((u^\epsi)_{\epsi}\)
is a bounded sequence in \(W^{1, \gamma} (\Tt^d)\)
by the Poincar\'e-Wirtinger inequality,
\eqref{uW1g1}, and \eqref{uW1g2}.
Therefore,  there is  \(u \in W^{1, \gamma} (\Tt^d)\)
satisfying, up to a
not relabeled further subsequence, \(u^\epsi
\weakly u\) weakly in \(  W^{1, \gamma} (\Tt^d)\)  as
\( \epsi
\to 0^+. \)  

Finally,  integrating the second equation
in (\ref{solmu=0})
over \(\Tt^d\) and letting $\epsi
\to 0^+$, we conclude that  \(\int_{\Tt^d}
m \,\dx =1\).
This completes Step~1.

\uline{\textit{Step 2.}} In this step, we show that \((m,u)\) satisfies \eqref{ws}.

We apply  a variation of  Minty's device (see, for instance, \cite{Eva, KiSt00}). Let $(\eta, v) \in {C^\infty
}(\Tt^d)
\times
{C^\infty}(\Tt^d)$ with $\inf_{\Tt^d}\eta>0$.
By Lemma~\ref{Lem:Ae mon} and \eqref{solmu=0},
we have 
{\setlength\arraycolsep{0.5pt}%
\begin{eqnarray*}
\begin{aligned}
0&\leq  \bigg(A^\epsi
\begin{bmatrix}\eta \\
v \\
\end{bmatrix} 
, \begin{bmatrix}\eta \\
 v
\end{bmatrix} 
-  
\begin{bmatrix}m^\epsi \\
u^\epsi \\
\end{bmatrix}  
\bigg)_{L^2(\Tt^d)\times L^2(\Tt^d)} \\
& = \int_{\Tt^d} \Big(-v-
H(x, Dv, D^2 v, \eta, h(\bfeta))\Big)(\eta - m^\epsi)
\, \dx\\
&\quad + \int_{\Tt^d} \Big(\epsi_1(\eta
+\Delta^{2p}\eta) +\epsi_2(\eta -\Delta \eta)+ \beta_{\epsi_1}(\eta) \Big) (\eta - m^\epsi)
\, \dx\\ 
&\quad + \int_{\Tt^d} \Big(\eta - \div(\eta D_p H (x, Dv, D^2 v, \eta, h(\bfeta)))\\
 &\hskip20mm + \big(\eta
D_{M_{ij}} H \big(x, Dv, D^2 v, \eta, h(\bfeta))  )_{x_i x_j} -1\Big) (v -
u^\epsi)
\, \dx\\  
 &\quad+\int_{\Tt^d} \Big( \epsi_1(v
+\Delta^{2p}v) +\epsi_2(v -\Delta v)\Big) (v
-
u^\epsi)
\, \dx.
\end{aligned}
\end{eqnarray*}}%
Letting $\epsi\to0^+$ and using \eqref{H:smooth},
\eqref{h}, the convergence results from Step~1, and the convergence $ \epsi_1(\eta
+\Delta^{2p}\eta) +\epsi_2(\eta -\Delta \eta)+ \beta_{\epsi_1}(\eta) +\epsi_1(v
+\Delta^{2p}v)+\epsi_2(v -\Delta v)  \to 0$
uniformly in $\Tt^d$ as \(\epsi \to 0^+\), which holds since $\inf_{\Tt^d}\eta>0$,
we obtain
{\setlength\arraycolsep{0.5pt}
\begin{eqnarray}\label{almostweak}
\begin{aligned}
& \int_{\Tt^d}
\Big[-v-
H(x, Dv, D^2 v, \eta, h(\bfeta))  \Big] (\eta - m)
\, \dx\\
&\qquad + \int_{\Tt^d} \big[\eta - \div(\eta
D_p H (x, Dv, D^2 v, \eta, h(\bfeta))\\
 &\hskip20mm + \big(\eta
D_{M_{ij}} H \big(x, Dv, D^2 v, \eta, h(\bfeta)
 )_{x_i x_j} -1  \big] (v -
u)
\, \dx \geq 0.
\end{aligned}
\end{eqnarray}}%
This concludes the proof in the  \(\EE
= \RR^+\) case because \((m,u) \in L^1(\Tt^d) \times  L^1(\Tt^d)\); thus, 
the left-hand side of \eqref{almostweak} coincides
with 
\begin{equation*}
\left\langle  \begin{bmatrix}\eta \\
v \\
\end{bmatrix}
 -  \begin{bmatrix}m \\
u \\
\end{bmatrix}, A \begin{bmatrix}\eta \\
v \\
\end{bmatrix} \right\rangle_{\mathcal{D}'(\Tt^d)
\times
\mathcal{D}'(\Tt^d), C^\infty(\Tt^d) \times
C^\infty(\Tt^d)}. 
\end{equation*}

In the  \(\EE=\RR^+_0\) case, we proceed as follows. Let  $\tilde \eta\in {C^\infty
}(\Tt^d)$ with \(\tilde\eta \geq 0\) and for \(\delta>0\),
set \(\eta_\delta := \tilde \eta +\delta\); using \eqref{almostweak} with
\(\eta\) replaced by \(\eta_\delta\) and letting
$\delta\to0^+$, by \eqref{H:smooth}, \eqref{hFD}, and \eqref{SobEmb},  we conclude that \eqref{almostweak}
holds for all $(\eta, v) \in {C^\infty
}(\Tt^d)
\times
{C^\infty}(\Tt^d)$ with $\eta\geq0$.
This concludes Step~2, as well as the proof of Theorem~\ref{Thm:main}.
\end{proof}

We end this section with a corollary to
the previous results.

\begin{corollary}
\label{Cor:limitws} Suppose that Assumptions \eqref{h}, \eqref{g}, \eqref{G:g
ei}, and \eqref{bddH}--\eqref{H:mon} are satisfied.
Then, there exist \((\bfm,u)
\in \Mm_{ac}(\Tt^d) \times W^{1,\gamma} (\Tt^d)
 \) with \(\int_{\Tt^d} m \,\dx = 1\) and a sequence
of positive numbers, \((\epsi_j)_{j\in\NN}\), convergent
to zero such
that:
\begin{enumerate}
\item[(i)] \((m,u)\) is a weak solution
  to \eqref{P}.
\item[(ii)] \((m,u)\) is the weak limit
in \(L^1(\Tt^d) \times W^{1,\gamma} (\Tt^d)\)
of a sequence \((m_{\epsi_j}, u_{\epsi_j})_{j\in\NN}
\subset D(A^{\epsi_j})\),
where each
\((m_{\epsi_j}, u_{\epsi_j})\) satisfies \eqref{RP}
with \(\epsi = \epsi_j\).
\end{enumerate} 
\end{corollary}

\begin{proof} 

 We established (ii) in the first step of the proof of Theorem ~\ref{Thm:main} and (i) in the second step of the same proof.
\end{proof}

\section{Properties of weak solutions}
\label{SOE}

Here, we examine some properties of weak solutions. 
To illustrate our methods, we consider the degenerate diffusion case, which corresponds to Hamiltonians
of the form~\eqref{degel} below. 
We show that the weak solutions given by Theorem \ref{Thm:main} are subsolutions of the first equation in \eqref{P}, the Hamilton-Jacobi equation.
Next, we recover an analog to the second equation in \eqref{P}, the Fokker-Planck equation.  
Finally, we show that 
in the set where \(m\) is positive,
the Hamiton-Jacobi equation in \eqref{P}
holds in a relaxed sense.

The main assumption of this section is the following.

\begin{hypotheses}{H}\setcounter{enumi}{\thehypH}
\item \label{dege}  
The Hamiltonian, $H$, can be written as 
\begin{equation}
\label{degel}
H(x, p, M, m, \theta)=H_0(x, p, m, \theta)-\sum_{i,j=1}^d
a_{ij}(x)M_{ij},
\end{equation}
where $H_0$ is a real-valued $C^\infty$ function
and $(a_{ij})_{1\leq i,j \leq d}$ is a 
\(C^\infty\) matrix-valued function. 
\begin{enumerate}
\item
\label{degelSUB}
Let $\gamma$ be as in 
Assumption~\eqref{bddH}. For all \(\varphi \in
C^\infty(\Tt^d)\) with \(\varphi \geq 0\), the
functional
\begin{equation*}
(m,u) \in L^1(\Tt^d) \times W^{1,\gamma}(\Tt^d)
\mapsto \intt \varphi\,  H_0(x,Du,m,h(\bfm))\,\dx
\end{equation*}
is sequentially weakly lower semicontinuous in
\(L^1(\Tt^d) \times W^{1,\gamma}(\Tt^d)\).

\item\label{FP}
Let $\gamma>1$ be as in 
Assumption~\eqref{bddH} with  $\tau=0$ and \(g\)
satisfying \eqref{G:sumg}. 
There exist constants, \(C_5>0\) and \(\alpha>0\),
such that,          
 for all $(x,p,m,\theta) \in        
     \Tt^d \times \RR^d  \times\EE \times \RR$,
     
\begin{enumerate}
\item \label{FPi}     
\( \displaystyle
|D_pH_0(x, p, m, \theta)|\leq C_5(1+|p|^{\gamma-1});
\)

\item \label{FPii} \(\displaystyle
g_1(x, m, \theta)\geq \frac{1}{C_5} m^\alpha-C_5.
\)
\end{enumerate} 

\item
\label{degelSUP}
Let \(\gamma>1\) and \(\alpha>0\) be as in Assumption~\eqref{FP}.
Let  \(L_0\) be the   Lagrangian 
associated with \(H_0\), defined for
 $(x,v,m,\theta) \in        
     \Tt^d \times \RR^d  \times\EE \times \RR$
by
\[
 L_0(x,v,m, \theta)=\sup_{p\in \RR^d} \left\{-p\cdot
v-H_0(x,p,m,
\theta) \right\}. 
\]
For all \((x,m,\theta) \in        
     \Tt^d   \times\EE \times \RR\), the map $p\mapsto
H_0 (x,p, m, \theta)$
is convex; moreover, the functional
\begin{equation*}
(m,J) \in L^{\alpha +1}(\Tt^d) \times L^\frac{(\alpha
+ 1) \gamma}{(\alpha
+ 1) \gamma-\alpha}(\Tt^d;\RR^d)
\mapsto \intt \,  mL_0 \left(x,\frac{J}{m},m,h(\bfm)\right)\dx
\end{equation*}
is sequentially weakly lower semicontinuous in
\(L^{\alpha +1}(\Tt^d) \times L^\frac{(\alpha
+ 1) \gamma}{(\alpha
+ 1) \gamma-\alpha}(\Tt^d;\RR^d)\).
\end{enumerate}
\setcounter{hypH}{\arabic{enumi}} 
\end{hypotheses}

We observe that the preceding assumption holds in the case given by \eqref{H's} with quadratic
Hamiltonian for $\tau=0$,  and with $g(m,\theta)=m^\alpha +\theta$, where $\alpha>0$. In general, Assumptions~\eqref{degelSUB}
and \eqref{degelSUP} depend on convexity properties
of the maps \((p,m,\theta) \in  \RR^d \times \EE \times \RR \mapsto H_0(x,p,m,
\theta)
\) and  \((v,m,\theta) \in  \RR^d \times \EE
\times \RR \mapsto m L_0\left(x, \frac{v}{m},m,
\theta\right)
\), respectively.

\subsection{Subsolution property}
\label{ssolp}
We begin by examining the Hamilton-Jacobi equation in \eqref{P}.
As previously stated, we suppose that  \(H\) satisfies \eqref{degel}. 
We say that $(\bfm,u)\in \Mm_{ac}(\Tt^d) \times
W^{1,\gamma} (\Tt^d)$ is a subsolution in the sense of distributions (i.e., in $\Dd'(\Tt^d)$)  of the Hamilton-Jacobi equation in \eqref{P} if for all $\varphi\in C^\infty(\Tt^d)$
with $\varphi\geq 0$, we have
\begin{equation}
\label{subsol}
\intt \left[\varphi (x)\left(u + H_0(x, Du, m, h(\bfm))\right)-\left(\varphi (x)a_{ij}(x)\right)_{x_ix_j}u\right]\dx\leq 0. 
\end{equation}
Next, we show
that the weak solutions given by Theorem \ref{Thm:main}
are subsolutions in $\Dd'(\Tt^d)$.

\begin{pro}
\label{subsolprop}
Suppose that Assumptions \eqref{h}, \eqref{g},
\eqref{G:g
ei},  \eqref{bddH}--\eqref{H:mon},  and \eqref{degelSUB} hold. 
Let $(m^\epsi, u^\epsi)\in D(A^\epsi)$ be a solution of \eqref{RP}. 
Let  \((m,u)
\in L^1(\Tt^d) \times W^{1,\gamma} (\Tt^d)
 \) with \(m \geq 0\) and \(\int_{\Tt^d} m \,\dx = 1\)  be a weak solution
of \eqref{P} obtained as a weak sublimit of \((m^\epsi,
u^\epsi)_\epsi\) in \(L^1(\Tt^d) \times W^{1,\gamma} (\Tt^d)\). Then, 
$(m,u)$ is a subsolution in the sense of distributions of the
Hamilton-Jacobi equation in \eqref{P}.       
\end{pro}

\begin{remark}
\label{onlimws}
We observe that there exist \((m^\epsi, u^\epsi)\) and \((m,u)\)
as in the statement of Proposition~\ref{ssolp}
 in view of Corollary~\ref{Cor:L=[0,1]}
with $\mu=0$ and  Corollary~\ref{Cor:limitws}.
\end{remark}

\begin{proof}
        To simplify the notation, we do not
distinguish  $(m^\epsi, u^\epsi)_\epsi$ from its
 subsequence that converges to \((m,u)\).

Take $\varphi \in C^\infty(\Tt^d)$, $\varphi\geq 0$. Multiplying the first equation 
in \eqref{RP} by $\varphi$ and integrating over $\Tt^d$ gives
\begin{align*}
&\intt \left[ \varphi \left(u^\epsi + H_0(x, Du^\epsi, m^\epsi, h(\bfm^\epsi))\right)-\left(
\varphi a_{ij}\right)_{x_ix_j}u^\epsi\right]\dx-\epsi_1\intt m^\epsi(\varphi+\Delta^{2p}\varphi) \,\dx\\
&\quad -\epsi_2\intt m^\epsi(\varphi-\Delta\varphi)\,\dx = \intt\beta_{\epsi_1}(m^\epsi)\varphi
\,\dx\leq 0 
\end{align*}
since $\beta_{\epsi_1}(\cdot)\leq 0$.      

Because $m^\epsi\rightharpoonup m$ in $L^1(\Tt^d)$, we have
\[
\lim_{\epsi\to 0^+} \left[\epsi_1\intt m^\epsi(\varphi+\Delta^{2p}\varphi)\,\dx +\epsi_2\intt m^\epsi(\varphi-\Delta\varphi)\,\dx\right] = 0. 
\]
Next, since \(u^\epsi \to u  \) in $L^\gamma(\Tt^d)$, we have   
\[
\lim_{\epsi\to 0^+} \intt \left[\varphi u^\epsi -(\varphi a_{ij})_{x_ix_j}u^\epsi
\right]
\dx
=
\intt  \left[\varphi
u -(\varphi a_{ij})_{x_ix_j}u
\right]
\dx.
\]
Finally,  Assumption~\eqref{degelSUB}, together
with the weak convergences $m^\epsi\rightharpoonup m$ in $L^1(\Tt^d)$ and $u^\epsi\rightharpoonup u$ in $W^{1,\gamma}(\Tt^d)$, gives
\[
\liminf_{\epsi\to 0^+}
\intt \varphi\, H_0(x, Du^\epsi, m^\epsi, h(\bfm^\epsi))\,\dx
\geq \intt \varphi\, H_0(x, Du, m, h(\bfm)) \,\dx.
\]
Consequently, we have \eqref{subsol}.   
\end{proof}

\subsection{Fokker-Planck equation}

To pursue
 our analysis further, we consider the case without congestion, $\tau=0$,  a natural growth condition on $D_pH_0$, and a power-like growth in $g_1$. This case corresponds to Assumption~\eqref{FP}. We begin by examining the integrability of the drift in the Fokker-Planck
equation.

\begin{pro}
\label{boundp}
Suppose that 
 Assumptions~\eqref{g}, \eqref{bddH},  and \eqref{FP} hold.
Let $q=\frac{(\alpha+1) \gamma}{(\alpha + 1) \gamma-\alpha}>1$.
Then, there exists a positive constant, $C$, 
such that, for any $\epsi\in (0,1)^2$
and any solution  $(u^\epsi, m^\epsi)\in D(A^\epsi)
$  to 
\eqref{RP}, we have
\[
\Vert m^\epsi\Vert_{L^{\alpha+1}(\Tt^d)} + \|m^\epsi D_pH(\cdot,Du^\epsi, m^\epsi,
 h(\bfm^\epsi))\|_{L^q(\Tt^d;\RR^d)}\leq C.
\]
\end{pro}
\begin{proof}

By \eqref{Est4}, Assumption~\eqref{FPii}, and Assumption~\eqref{G:g
ub} with \(\delta=1\),  $m^\epsi$
is bounded 
in $L^{\alpha + 1}(\Tt^d)$ independently of
\(\epsi\). In particular, because \(q<\alpha+1\), \((m^\epsi)_\epsi\)
is bounded in \(L^q (\Tt^d)\).
Then,
in view of Assumptions~\eqref{degel} and \eqref{FPi},
 it suffices to prove that   \(m^\epsi |Du^\epsi|^{\gamma-1}\)
is bounded in \(L^q (\Tt^d)\) independently of
\(\epsi\).

Set \(s=  {\gamma (\alpha +1)}\) and \( r
= \frac \gamma {\gamma -1}\). Observe that
\(m^\epsi |Du^\epsi|^{\gamma-1} = \left[( m^\epsi)^{\alpha +1} \right]^{\frac 1 s} \left(|Du^\epsi|^\gamma
m^\epsi \right)^{\frac 1 r}\) and \(\frac 1 q = \frac 1 r + \frac 1 s\).
As we proved before, \((m^\epsi)^{\alpha+1}\)
is bounded in $L^1(\Tt^d)$ independently of \(\epsi\).
Moreover, from Lemma \ref{Lem:ape} with \(\tau=0\), we have that $|Du^\epsi|^\gamma m^\epsi $
is bounded 
in $L^1(\Tt^d)$  independently of \(\epsi\).
The result then follows by H\"older's
inequality. 
\end{proof}

\begin{pro}
       \label{wfp}
Suppose that Assumptions \eqref{h}, \eqref{g},
\eqref{G:g
ei},  \eqref{bddH}--\eqref{H:mon},   and \eqref{FP}
hold. Let $(u^\epsi, m^\epsi)\in D(A^\epsi)$ be a solution of \eqref{RP}. Let  \((m,u)
\in L^1(\Tt^d) \times W^{1,\gamma} (\Tt^d)
 \) with \(m \geq 0\) and \(\int_{\Tt^d} m \,\dx
= 1\)  be a weak solution
to \eqref{P} obtained as a weak sublimit of \((m^\epsi,
u^\epsi)_\epsi\) in \(L^1(\Tt^d) \times W^{1,\gamma}
(\Tt^d)\).
Let $q=\frac{(\alpha+1) \gamma}{(\alpha+1) \gamma-\alpha}$
and set $J^\epsi=m^\epsi D_pH(x,Du^\epsi, m^\epsi, h(\bfm^\epsi))$. 
Then, there exists $J\in L^q(\Tt^d;\RR^d)$ such that, up to a not relabeled subsequence, $J^\epsi\rightharpoonup J$ in $L^q(\Tt^d;\RR^d)$. Moreover,
\begin{equation}
\label{wfpe}
m-\div(J)-(a_{ij}(x)m)_{x_ix_j}=1
\end{equation}
in $\Dd'(\Tt^d)$.
\end{pro}
\begin{proof}
By Proposition~\ref{boundp}, up to a not relabeled
subsequence, we have $J^\epsi\rightharpoonup J$ in $L^q(\Tt^d;\RR^d)$ for some $J\in L^q(\Tt^d;\RR^d)$. 
The result follows by convergence in \(\Dd'(\Tt^d)\)  together with  the second equation in \eqref{RP}.  
\end{proof}

\subsection{Supersolution property}

Here, we revisit the Hamilton-Jacobi equation in \eqref{P}. From the results in Subsection \ref{ssolp}, we know that weak solutions are subsolutions. Next, we prove the converse result; that is,  weak solutions solve,   in a relaxed sense, the Hamilton-Jacobi in the set where \(m\) is positive.
 
\begin{pro}
\label{supersolprop}
Suppose that Assumptions \eqref{h}, \eqref{g},
\eqref{G:g
ei},  \eqref{bddH}--\eqref{H:mon},    \eqref{FP}, and \eqref{degelSUP}
hold. Let $q=\frac{(\alpha+1) \gamma}{(\alpha+1) \gamma-\alpha}$.
Let $(u^\epsi, m^\epsi)$, $(u,m)$, and  $J$ be as in Proposition~\ref{wfp}. Then,  for all $\varphi\in C^\infty(\Tt^d)$, we have
\begin{equation}
\label{wsuper}
\intt 
\left[J D\varphi - H_0(x, D\varphi, m, h(\bfm))m-u \right] \dx\leq 0. 
\end{equation}
\end{pro}
\begin{remark}
We regard \eqref{wsuper} as a weak version of 
\begin{equation}
\label{wsuper2}
\intt m \left[-u-H_0(x, Du, m, h(\bfm))+a_{ij}(x)u_{x_i
x_j}\right]
\dx \leq 0. 
\end{equation}
To see this, we observe that if $u$ is regular enough, then
\eqref{wsuper} for
$\varphi=u$ yields  
\begin{equation}
\label{wsuperbis}
\intt 
\left[J Du - H_0(x, Du, m, h(\bfm))m-u \right]
\dx \leq 0.
\end{equation} 
Then, due to \eqref{wfpe}, we get 
\[
\intt 
(J D u - u) \, \dx =\intt m(-u+a_{ij}(x)u_{x_ix_j})
\,\dx.
\]
Combining the previous identity with  \eqref{wsuperbis}
gives \eqref{wsuper2}.
\end{remark}

\begin{proof}[Proof of Proposition~\ref{supersolprop}]
To simplify the notation, we do not relabel
the weakly convergent subsequence of \((m^\epsi, u^\epsi,J^\epsi)_\epsi\)
to \((m,u,J) \) in \(L^1(\Tt^d) \times W^{1,\gamma}(\Tt^d)
\times L^q(\Tt^d;\RR^d)\). Note that in view of
Proposition~\ref{boundp}, we also have \(m^\epsi
\weakly m\) weakly in \(L^{\alpha +1}(\Tt^d)\).

Next, we  observe that 
\begin{align*}
&\intt m^\epsi \left[-u^\epsi-H_0(x, Du^\epsi, m^\epsi, h(\bfm^\epsi))+a_{ij}(x)u^\epsi_{x_ix_j}\right
] \dx\\
&+\intt m^\epsi \left[\epsi_1(m^\epsi
+\Delta^{2p}m^\epsi) +\epsi_2(m^\epsi -\Delta m^\epsi)+ 
\beta_{\epsi_1}(m^\epsi)\right] \dx\\
&+ \intt u^\epsi \left[ m^\epsi - \div(m^\epsi D_p H_0) - \big(m^\epsi
a_{ij}\big)_{x_i x_j} -1 + \epsi_1(u^\epsi
+\Delta^{2p}u^\epsi) +\epsi_2(u^\epsi -\Delta u^\epsi)\right]\dx = 0.
\end{align*}
Integrating  the terms multiplied by $\epsi_1$ and $\epsi_2$ by parts, 
and taking into account their sign, we conclude
that\begin{align*}
&\intt m^\epsi \left[-u^\epsi-H_0(x, Du^\epsi, m^\epsi, h(\bfm^\epsi))+a_{ij}(x)u^\epsi_{x_ix_j}+
\beta_{\epsi_1}(m^\epsi) \right]\dx\\
&+ \intt u^\epsi \left[ m^\epsi - \div(m^\epsi D_p H_0) - \big(m^\epsi
a_{ij}\big)_{x_i x_j} -1 \right]\dx \leq 0.
\end{align*}    
By Proposition~\ref{punifbd}, 
we have that $\|\beta_{\epsi_1}(m^\epsi)\|_{L^1(\Tt^d)}$ is bounded independently of \(\epsi\). Accordingly, we find
\[
\lim_{\epsi\to0^+}\intt m^\epsi \beta_{\epsi_1}(m^\epsi) \,\dx= 0
\] 
because \(|m^\epsi \beta_{\epsi_1}(m^\epsi)| \leq
\epsi_1|\beta_{\epsi_1}(m^\epsi)|\) by definition
of \(\beta_{\epsi_1}(\cdot)\).    
Hence,
\begin{align*}
        &
        \liminf_{\epsi\to 0^+} \bigg(\intt m^\epsi \left(-H_0(x, Du^\epsi , m^\epsi , h(\bfm^\epsi ))\right) \dx
+ \intt u^\epsi \left[ - \div(m^\epsi  D_p H_0)  -1 \right]\dx \bigg) \leq 0.
\end{align*}

Recalling that $J^\epsi=m^\epsi D_p H_0(x, Du^\epsi , m^\epsi , h(\bfm^\epsi ))$, 
the definition of the Legendre transform and Assumption~\eqref{degelSUP}
give
\[
 \liminf_{\epsi \to 0^+}
\intt m^\epsi L_0\left(x,\frac{-J^\epsi}{m^\epsi}, m^\epsi, h(\bfm^\epsi)\right) \dx\leq \lim_{\epsi\to0^+}
\intt u^\epsi\,\dx = \intt u\,\dx,
\]
where we used the convergence \(u^\epsi \to u\)
in \(L^\gamma(\Tt^d)\). Then, because \(-J^\epsi \weakly -J\) in \(L^q(\Tt^d;\RR^d)\)
and \(m^\epsi \weakly m\) in \(L^{\alpha+1}(\Tt^d)\),  Assumption~\eqref{degelSUP} implies
that\begin{equation}
\label{ssproof1}
\intt m L_0\left(x,\frac{-J}{m}, m, h(\bfm)\right)
\dx\leq  \intt u\,\dx.
\end{equation}
On the other hand, using  the definition of the Legendre transform again, we have
\begin{equation}
\label{ssproof2}
m L_0\left(x,\frac{-J}{m}, m, h(\bfm)\right)\geq J D\varphi - H_0(x, D\varphi, m, h(\bfm)) m
\end{equation}
  for any $\varphi\in C^\infty(\Tt^d)$.
Consequently, \eqref{ssproof1} and \eqref{ssproof2} yield
\eqref{wsuper}.
\end{proof}

\section{Improved regularity}
\label{Sect:eaa}

In this section, we consider the quadratic case for either first-order problems or elliptic
problems with constant coefficients, encoded
in the Assumption~\eqref{CC} below. We  get improved regularity for $m$
and $u$, including higher integrability and Sobolev
estimates for $m$. Further, we show that the Fokker-Plank
equation holds pointwise, not just in the sense
of distributions.
Finally, we prove that the Hamilton-Jacobi equation
holds pointwise in the set where $m$ is positive.

The main assumption of this section is the following.

\begin{hypotheses}{H}\setcounter{enumi}{\thehypH}
        \item \label{CC}      
The Hamiltonian, $H$, is of the form
{\setlength\arraycolsep{0.5pt}%
        \begin{eqnarray}\label{Hquad}
        \begin{aligned}
        H(x,p,M,m, \theta) =  \frac{|p|^2}{2}
- V(x, m,\theta) - \sigma^2  \sum_{i=1}^d M_{ii},
        \end{aligned}
        \end{eqnarray}
}%
where $\sigma \in \RR$ and  $V:\Tt^d\times \EE\times
\Rr\to\Rr$ is a $C^\infty$ function.
\begin{enumerate}
\item \label{V:H1-H3} There exist  constants,
\( c_1>0\), \(c_2>0\), and \(c_3>0\), and  a 
continuous function, $g:\Tt^d \times \EE \times\RR\to\RR$,
 such that:  
        \begin{enumerate}
        \item \label{V:lb}  For all $(x,m,\theta)
\in     \Tt^d  \times\EE \times \RR$, 
        \[V(x,m,\theta) \geq c_1 g(x,m,\theta)
        - \frac 1 {c_2}.\]
        \item  For all $\bfm  \in
\Mm_{ac}(\Tt^d)$, 
\begin{equation*}
\intt V(x ,m, h(\bfm))\,\dx
        \leq c_2  + c_3\intt g(x,m, h(\bfm))\,\dx.
\end{equation*}
        \item  For all \(\eta, \, \bar\eta   \in
{C^\infty}(\Tt^d; \EE)\),
  \[ \intt \left[ V(x, \eta, h(\bfeta)) -V(x,
  \bar \eta, h(\bar \bfeta))\right](\eta - \bar
\eta)\,\dx
  \geq 0.
  \]
        \end{enumerate}
\item \label{V:extra}
There exist  constants, \(\alpha>0\) and \(\kappa_1>0\),
 such that: 
\begin{enumerate}
\item  \label{V:m^a+1}
For all $(x,m,\theta) \in        
     \Tt^d  \times\EE \times \RR$, the function
\(g\) in \eqref{V:H1-H3} satisfies \eqref{G:sumg}
        with 
        \[g_1(x,m, \theta) \geq \frac 1{\kappa_1}
m^{\alpha}-{\kappa_1}.\]
\item  \label{V:Dm^a+1}
For all \(\eta
  \in {C^\infty}(\Tt^d; \EE)\),
\[
\intt 
\left(V(x,\eta,h(\bfeta)
)\right)_{x_i}\eta_{x_i}
\,\dx\geq
\frac 1{\kappa_1}\intt \eta^{\alpha-1}|D\eta|^2
\,\dx -
{\kappa_1}.  
\] 
\end{enumerate}

\item \label{V:Fatou}
Let \(\alpha>0\) and \(\kappa_1>0\)
be as in \eqref{V:extra}. For all $(x,m,\theta)
\in        
     \Tt^d  \times\EE \times \RR$, we have 
\begin{equation*}
V(x,m,\theta) \leq \kappa_1 (1 + m^\alpha + |\theta|).
\end{equation*}
\end{enumerate}
\end{hypotheses}

We observe that a Hamiltonian  satisfying
\eqref{V:H1-H3} is a particular case of a Hamiltonian
for which \eqref{bddH}--\eqref{H:mon} hold (with
\(\gamma
= 2\) and \(\tau=0\)).

In some cases, we will also require the operator \(h\)  to satisfy
the following condition:
\begin{hypotheses}{h}\setcounter{enumi}{\thehyph}
\item \label{hh}
There exists a positive constant, \(C\), such that,
for all \(\bfm\), \(\bar \bfm \in \Mm_{ac}(\Tt^d)\),
we have 
\begin{equation*}
\Vert h(\bfm) - h(\bar\bfm) \Vert_{L^1(\Tt^d)}
\leq C \Vert m - \bar m \Vert_{L^1(\Tt^d)}.
\end{equation*}

\end{hypotheses}

Examples of operators  satisfying the
Assumptions~\eqref{h} and \eqref{hh} are those in \eqref{H's}.

For Hamiltonians as in \eqref{Hquad}, \eqref{P} takes the form
\begin{equation}\label{Pquad}
\begin{cases}
\displaystyle-u- \frac{|Du|^2}{2} + V(x, m,h(\bfm)) + \sigma^2 \Delta u =0,\\ 
m - \div(m Du) - \sigma^2 \Delta m -1  =0.
\end{cases}
\end{equation}


First, we consider the regularized problem \eqref{RP} or, equivalently,  \eqref{Pe} with $\mu=0 $ and prove additional
estimates
that are  uniform in \(\epsi \in (0,1)^2\).

\begin{lemma}\label{Lem:moreape}
Suppose that Assumptions~\eqref{g}, \eqref{V:H1-H3},
and \eqref{V:extra} hold.
Fix $\epsi \in (0,1)^2$. Let $(m^\epsi,u^\epsi)\in
C^\infty(\Tt^d)\times C^\infty(\Tt^d)$ with
$\inf_{\Tt^d}m^\epsi>0$
 in $\Tt^d$ be a solution of \eqref{RP}.
 Then, there exists a positive constant,
\(C\), independent of \(\epsi\)  such that
{\setlength\arraycolsep{0.5pt}%
\begin{eqnarray*}
\begin{aligned}
&\Vert
u^\epsi\Vert_{W^{1,2}(\Tt^d)} + 
\Vert (m^\epsi)^{\frac{\alpha+1}{2}} \Vert_{W^{1,2}(\Tt^d)}
+\Vert \beta_{\epsi_1} (m^\epsi)\Vert_{L^1(\Tt^d)}+
\Vert( m^\epsi)^{\frac{1}{2}} Du^\epsi
\Vert_{L^2(\Tt^d;\mathbb{R}^{d})}  \\
&+ 
\Vert( m^\epsi)^{\frac{1}{2}} D^2 u^\epsi
\Vert_{L^2(\Tt^d;\mathbb{M}_{d\times d})}+
\Vert (m^{\epsi})^{\frac{\alpha+1}{2}}
Du^{\epsi}
\Vert_{W^{1,1}(\Tt^d;\RR^d)} +\Vert\sqrt{\epsi_1} u^{\epsi} 
\Vert_{W^{2p+1,2}(\Tt^{d})}  \\
& + \Vert\sqrt{\epsi_1}
m^{\epsi}\Vert_{W^{2p+1,2}(\Tt^{d})}
+\Vert\sqrt{\epsi_2} u^{\epsi}\Vert_{W^{2,2}(\Tt^{d})} + \Vert\sqrt{\epsi_2}
m^{\epsi}\Vert_{W^{2,2}(\Tt^{d})} 
\leq C.
\end{aligned}
\end{eqnarray*}}%
\end{lemma}

\begin{proof}
Here, to simplify the notation, we write \(\bfm\),
$m$, and $u$ in place of \(\bfm^\epsi\), $m^\epsi$, and $u^\epsi$,
respectively.
Moreover, $C$  represents a positive constant independent of $\epsi$
 and whose value
may change from one  line to another.
Also, we do not use the Einstein convention
on repeated indices, and we specify all the sums.

Since $(m,u)$ is a solution to \eqref{RP} and \(H\) is given by \eqref{Hquad}, we have
{\setlength\arraycolsep{0.5pt}%
\begin{eqnarray}
&&\displaystyle -u-
\frac{|Du|^2}{2}
+  V(x, m , h(\bfm))  + \sigma^2 \Delta u+ \epsi_1(m
+\Delta^{2p}m) + \beta_{\epsi_1}(m)+
\epsi_2(m -\Delta m)=0,\qquad\label{eqnme0}\\
&&\displaystyle m - \div(m Du) - \sigma^2 \Delta
m-1 + \epsi_1(u
+\Delta^{2p}u) + \epsi_2(u -\Delta u)= 0, \label{eqnue0}
\end{eqnarray}}%
pointwise in \(\Tt^d\).

Next, we recall that under Assumption~\eqref{V:H1-H3}, the Hamiltonian $H$ in \eqref{Hquad} satisfies Assumption
\eqref{bddH} with \(\gamma=2\),  $\tau=0$, and $g$
given by \eqref{V:H1-H3}. Because \(g\)
also satisfies \eqref{g},
 Proposition~\ref{punifbd} 
 and Lemma~\ref{Lem:ape}  with \(\gamma=2\),  $\tau=0$, and \(\mu=0 \) give the bounds
\begin{equation}
\label{bdss000}
\Vert Du\Vert_{L^2(\Tt^d;\RR^d)} + \bigg| \intt
u\,\dx\bigg|
\leq C,
\end{equation}
\begin{equation}
\label{bddss002}
\intt m g_1(x, m , h(\bfm))\,\dx \leq C,
\end{equation}
and
\begin{equation*}
\Vert \beta_{\epsi_1} (m)\Vert_{L^1(\Tt^d)}+\Vert m^{\frac{1}{2}}
Du \Vert_{L^2(\Tt^d;\mathbb{R}^{d})}
 \leq C.
\end{equation*}
The estimate \eqref{bdss000} together with the Poincar\'e-Wirtinger
inequality  yields
{\setlength\arraycolsep{0.5pt}%
\begin{eqnarray}
 \Vert
u\Vert_{W^{1,2}(\Tt^d)} \leq C.
\label{bds02}
\end{eqnarray}}%
Combining  \eqref{bddss002}, \eqref{V:m^a+1},
and 
\eqref{G:g ub} with \(\delta=1\), 
we get
{\setlength\arraycolsep{0.5pt}%
\begin{eqnarray}
\label{ma+1inL2}
 \Vert m^{\frac{\alpha+1}{2}}
\Vert_{L^{2}(\Tt^d)}\leq C.
\end{eqnarray}}%
After differentiating \eqref{eqnme0} twice
with respect to \(x_j\), \(j\in\{1,...,d\}\),
multiplying  by $m$, and  integrating over \(\Tt^d\),
the resulting equality becomes
{\setlength\arraycolsep{0.5pt}%
\begin{eqnarray*}
\begin{aligned}
&-\int_{\Tt^d} u_{x_j x_j} m \,\dx - \int_{\Tt^{d}}
\sum_{i=1}^d |u_{x_i x_j}|^2 m \,\dx + \int_{\Tt^{d}}
 \sum_{i=1}^d (m u_{x_i})_{x_i}  u_{x_j x_j}\,\dx
\\&\quad - \int_{\Tt^d} (V(x,m,h(\bfm))_{x_j}m_{x_j} \,\dx\\
& \quad+ \sigma^2 \int_{\Tt^d} \Delta u\, m_{x_jx_j} \,\dx -\epsi_1 \int_{\Tt^d}  |m_{x_{j}}|^2
\,\dx - \epsi_1 \int_{\Tt^d}  |\Delta^pm_{x_{j}}|^2
\,\dx \\
&\quad- \int_{\Tt^d} \beta_{\epsi_1}'(m) |m_{x_{j}}|^2
\,\dx-\epsi_2 \int_{\Tt^d}  |m_{x_{j}}|^2
\,\dx - \epsi_2 \int_{\Tt^d}  |Dm_{x_{j}}|^2
\,\dx =0
\end{aligned}
\end{eqnarray*}}%
using integration by parts. In
view of \eqref{eqnue0}, we have that
{\setlength\arraycolsep{0.5pt}%
\begin{eqnarray*}
\begin{aligned}
\sum_{i=1}^d (m u_{x_i})_{x_i}  = m -\sigma^2 \Delta
m -1 +\epsi_1
u + \epsi_1 \Delta^{2p} u +\epsi_2u - \epsi_2
\Delta u.
\end{aligned}
\end{eqnarray*}}%
Using this identity in the preceding 
equality, summing over \(j\in\{1,...,d\},\)
and using \eqref{V:Dm^a+1},
we conclude that
{\setlength\arraycolsep{0.5pt}
\begin{eqnarray}
&&\int_{\Tt^{d}}
 |D^2u|^2 m \,\dx +\frac{1}{\kappa_1} \int_{\Tt^{d}}  m^{\alpha-1}
|Dm|^2\,\dx + \epsi_1 \int_{\Tt^d}  |Du|^2
\,\dx + \epsi_1 \int_{\Tt^d}  |\Delta^p Du|^2
\,\dx\nonumber \\
&&\quad+\epsi_1 \int_{\Tt^d}  |Dm|^2
\,\dx + \epsi_1 \int_{\Tt^d}  |\Delta^pDm|^2
\,\dx +\epsi_2 \int_{\Tt^d}  |Du|^2
\,\dx + \epsi_2 \int_{\Tt^d}  |D^2u|^2
\,\dx \nonumber\\ 
&&\quad  +\epsi_2 \int_{\Tt^d}  |Dm|^2
\,\dx + \epsi_2 \int_{\Tt^d}  |D^2m|^2
\,\dx +\int_{\Tt^d} \beta_{\epsi_1}'(m) |Dm|^2
\,\dx \leq {\kappa_1}. \label{bds03}
\end{eqnarray}}%

Because \(|D(m^{\frac{\alpha+1}{2}})|^2
=\frac{(\alpha+1)^2}{4} m^{\alpha-1}
|Dm|^2\) and because all terms
on the left-hand side of the inequality in \eqref{bds03}
are non-negative, it follows  that
{\setlength\arraycolsep{0.5pt}%
\begin{eqnarray}
\begin{aligned}
\Vert m^{\frac{1}{2}} D^2 u
\Vert_{L^2(\Tt^d;\mathbb{M}_{d\times d})} +
 \Vert m^{\frac{\alpha+1}{2}} 
 \Vert_{W^{1,2}(\Tt^d)}
 \leq C,
\label{bds04}
\end{aligned}
\end{eqnarray}}%
where we also take into account \eqref{ma+1inL2}.
Moreover, invoking the Poincar\'e-Wirtinger inequality
once more and using  integration
by parts, we get
{\setlength\arraycolsep{0.5pt}%
\begin{eqnarray*}
&&\Vert\sqrt{\epsi_1} u
\Vert_{W^{2p+1,2}(\Tt^{d})} + \Vert\sqrt{\epsi_1}
m\Vert_{W^{2p+1,2}(\Tt^{d})} + \Vert\sqrt{\epsi_2} u 
\Vert_{W^{2,2}(\Tt^{d})} + \Vert\sqrt{\epsi_2}
m\Vert_{W^{2,2}(\Tt^{d})}  \leq C.
\end{eqnarray*}}%

Finally, we observe that 
{\setlength\arraycolsep{0.5pt}%
\begin{eqnarray*}
\Vert
m^{\frac{\alpha+1}{2}}
Du
\Vert_{W^{1,1}(\Tt^d;\RR^d)}\leq C
\end{eqnarray*}}%
is a consequence of \eqref{bds02} and \eqref{bds04}.
\end{proof}%

The next result refines Theorem~\ref{Thm:main} and improves the result in Proposition \ref{wfp}. 
\begin{theorem}\label{Thm:quad}
Suppose that Assumptions~\eqref{h}, \eqref{hh}, \eqref{g},
and \eqref{CC} hold.
 Then, there exists a weak solution, $(m,u)\in
\mathcal{D}'(\Tt^d) \times \mathcal{D}'(\Tt^d)$,
with $m\geq 0$   to \eqref{Pquad}. Moreover, \((m^{\frac{\alpha + 1}{2}},u)
\in W^{1,2}(\Tt^d) \times W^{1,2} (\Tt^d)
 \),  \(\int_{\Tt^d} m \,\dx = 1\), 
\begin{equation}
\label{72}
\displaystyle-u- \frac{|Du|^2}{2} + V(x, m, h(\bfm)) + \sigma^2 \Delta u
\geq 0
\end{equation}
 in the sense of distributions, and
\begin{equation}
\label{72'}
m - \div(m Du) - \sigma^2 \Delta m -1  =0
\end{equation}
 \aev\ in \(\Tt^d\).
\end{theorem}

\begin{proof}
By Corollary~\ref{Cor:L=[0,1]},
for each $\epsi=(\epsi_1, \epsi_2) \in (0,1)^2$, there exists  $(m^\epsi, u^\epsi) \in
C^\infty(\Tt^d)\times C^\infty(\Tt^d)$ with
$\inf_{\Tt^d}m^\epsi>0$
 satisfying \eqref{eqnme0}--\eqref{eqnue0}.
 
We assume that \(\epsi_1\) and \(\epsi_2\) take values on fixed sequences of positive numbers converging to zero. We proceed with the proof in four steps.

\textit{Step 1.} Here, we establish
preliminary
convergences of \((m^\epsi, u^\epsi)\) first as \(\epsi_1
\to 0^+,\) and then as \(\epsi_2 \to 0^+\).

We recall  that the embeddings 
\(W^{1,2}(\Tt^d) \hookrightarrow L^q(\Tt^d)\) and \(W^{2,2}(\Tt^d) \hookrightarrow L^r(\Tt^d)\)  are compact for
all \(q\in [1,2^*) \) and \(r\in [1,(2^*)^*) \) and  continuous for
all \(q\in [1,2^*] \) and \(r\in [1,(2^*)^*]
\). 
Hence, the uniform bounds
in  Lemma~\ref{Lem:moreape} with respect to \(\epsi_1\) give  that, for each \(\epsi_2\),
there exist \(u^{\epsi_2}\in W^{2,2}(\Tt^d) \)
and \(m^{\epsi_2} \in W^{2,2}(\Tt^d)\)
 with \((m^{\epsi_2})^{\frac{\alpha +1}{2}}\in W^{1,2}(\Tt^d)
\) such that, up to a not relabeled subsequence  of \((m^\epsi, u^\epsi)_{\epsi_1}\),
{\setlength\arraycolsep{0.5pt}
\begin{eqnarray}
&& u^\epsi \weakly u^{\epsi_2} \hbox{ weakly in } W^{2,2}(\Tt^d) \hbox{ as } \epsi_1 \to 0^+, \label{ue wrt e1}\\
&& m^\epsi \weakly m^{\epsi_2} \hbox{ weakly
in } W^{2,2}(\Tt^d) \hbox{ as } \epsi_1 \to
0^+, \label{me wrt e1}\\
&& (m^\epsi)^{\frac{\alpha +1}{2}} \weakly (m^{\epsi_2})^{\frac{\alpha +1}{2}} \hbox{ weakly
in } W^{1,2}(\Tt^d) \hbox{ as } \epsi_1 \to
0^+, \label{mea wrt e1}\\
&& \sqrt{\epsi_1} u^\epsi \to 0 \hbox{ in } W^{2p,2}(\Tt^d) \hbox{ as } \epsi_1 \to
0^+, \label{sue wrt e1}\\
&&  \sqrt{\epsi_1} m^\epsi
\to 0 \hbox{ in } W^{2p,2}(\Tt^d) \hbox{ as } \epsi_1 \to
0^+, \label{sme wrt e1}\\
&& (m^\epsi)^{\frac{1}{2}} Du^\epsi \weakly
(m^{\epsi_2})^{\frac{1}{2}} Du^{\epsi_2} \hbox{
weakly
in } L^2(\Tt^d;\RR^d ) \hbox{ as } \epsi_1 \to
0^+, \label{smeDue wrt e1}\\
&& (m^\epsi)^{\frac{1}{2}} D^2u^\epsi \weakly
(m^{\epsi_2})^{\frac{1}{2}} D^2u^{\epsi_2} \hbox{
weakly
in } L^2(\Tt^d;\mathbb{M}_{d\times d}) \hbox{ as } \epsi_1 \to
0^+, \label{smeD2ue wrt e1}\\
&& (m^\epsi)^{\frac{\alpha +1}{2}} Du^\epsi \weakly
(m^{\epsi_2})^{\frac{\alpha +1}{2}} Du^{\epsi_2} \hbox{
weakly
in } W^{1,1}(\Tt^d;\RR^d ) \hbox{ as } \epsi_1 \to
0^+, \label{meaDue wrt e1}\\
&& h(\bfm^\epsi) \to h(\bfm^{\epsi_2}) \hbox{ in }
L^1(\Tt^d) \hbox{ as } \epsi_1 \to
0^+, \label{hme wrt e1}
\end{eqnarray}}%
where we also used the Rellich--Kondrachov theorem,  the uniqueness of weak limits, and Assumption~\eqref{hh}. Observe that
\eqref{meaDue wrt e1} holds weakly in \(W^{1,1}(\Tt^d;\RR^d )\), and not only weakly-$\star$ in \(BV(\Tt^d;\RR^d
)\), by \eqref{ue wrt e1}, \eqref{mea wrt e1},
and the Rellich--Kondrachov theorem. Moreover, in view of 
\eqref{ue wrt e1}, \eqref{me wrt e1}, and \eqref{hme wrt e1}, we can
assume that the subsequence also satisfies,  for \aev\ \(x\in \Tt^d\),
{\setlength\arraycolsep{0.5pt}
\begin{eqnarray}
&& Du^{\epsi} (x) \to Du^{\epsi_2}(x), \enspace
  m^{\epsi}
(x) \to m^{\epsi_2}(x), \hbox{ and }  h(\bfm^{\epsi})
(x) \to h(\bfm^{\epsi_2})(x) \hbox{ as } \epsi_1
\to 0^+.  \qquad\label{pointwise e1}
\end{eqnarray}}%
Note also that by \eqref{me wrt e1}, 
\((m^\epsi)^{\frac{1}{2}} \to 
(m^{\epsi_2})^{\frac{1}{2}}\) in \(L^2(\Tt^d)\) as \(\epsi_1 \to 0^+\), which, together with \eqref{smeDue wrt e1} and \eqref{smeD2ue wrt e1}, yields
{\setlength\arraycolsep{0.5pt}
\begin{eqnarray}
&& m^\epsi Du^\epsi \weakly m^{\epsi_2} 
Du^{\epsi_2} \hbox{ weakly
in } L^1(\Tt^d;\RR^d ) \hbox{ as } \epsi_1
\to
0^+, \label{meDue wrt e1}\\
&& {m^\epsi D^2u^\epsi \weakly m^{\epsi_2} 
D^2u^{\epsi_2} \hbox{ weakly
in } L^1(\Tt^d;\RR^d ) \hbox{ as } \epsi_1
\to
0^+.} \nonumber
\end{eqnarray}}%

Additionally, using the sequential lower semicontinuity of the norm in Lebesgue and Sobolev spaces with respect to the corresponding weak convergence, it follows from Lemma~\ref{Lem:moreape} and \eqref{ue wrt e1}--\eqref{meaDue wrt e1}  that there is a positive constant, \(C\), independent of \(\epsi_2\) such that
{\setlength\arraycolsep{0.5pt}%
\begin{eqnarray*}
\begin{aligned}
&\Vert
u^{\epsi_2}\Vert_{W^{1,2}(\Tt^d)} + 
\Vert (m^{\epsi_2})^{\frac{\alpha+1}{2}} 
\Vert_{W^{1,2}(\Tt^d)} +
\Vert( m^{\epsi_2})^{\frac{1}{2}} Du^{\epsi_2}
\Vert_{L^2(\Tt^d;\mathbb{R}^{d})}  
+ \Vert( m^{\epsi_2})^{\frac{1}{2}} D^2 u^{\epsi_2}
\Vert_{L^2(\Tt^d;\mathbb{M}_{d\times d})}\\
&\quad+ 
\Vert (m^{{\epsi_2}})^{\frac{\alpha+1}{2}}
Du^{{\epsi_2}}
\Vert_{W^{1,1}(\Tt^d;\RR^d)}   
+\Vert\sqrt{\epsi_2} u^{\epsi_2}\Vert_{W^{2,2}(\Tt^{d})}
+ \Vert\sqrt{\epsi_2}
m^{\epsi_2}\Vert_{W^{2,2}(\Tt^{d})} 
\leq C.
\end{aligned}
\end{eqnarray*}}%
Thus, arguing as above,  there exist \(u\in W^{1,2}(\Tt^d) \cap L^{2^*}(\Tt^d)\)
and \(m \in L^{\frac{\alpha + 1}{2}2^*}(\Tt^d)\)
 with \(m^{\frac{\alpha +1}{2}}\in W^{1,2}(\Tt^d)
\) such that, up to a not relabeled subsequence of \((m^{\epsi_2}, u^{\epsi_2})_{\epsi_2}\),
{\setlength\arraycolsep{0.5pt}
\begin{eqnarray}
&& u^{\epsi_2} \weakly u \hbox{ weakly
in } W^{1,2}(\Tt^d) \hbox{ as } \epsi_2 \to
0^+, \label{ue wrt e2}\\
&& (m^{\epsi_2})^{\frac{\alpha +1}{2}} \weakly
m^{\frac{\alpha +1}{2}} \hbox{
weakly
in } W^{1,2}(\Tt^d) \hbox{ as } \epsi_2 \to
0^+, \label{mea wrt e2}\\
&& (m^{\epsi_2})^{\frac{1}{2}} Du^{\epsi_2} \weakly
m^{\frac{1}{2}} Du \hbox{
weakly
in } L^2(\Tt^d;\RR^d ) \hbox{ as } \epsi_2
\to
0^+, \nonumber
\\
&& (m^{\epsi_2})^{\frac{\alpha +1}{2}} Du^{\epsi_2}
\weaklystar
m^{\frac{\alpha +1}{2}} Du \hbox{
weakly-$\star$
in } BV(\Tt^d;\RR^d ) \hbox{ as } \epsi_2 \to
0^+, \label{meaDue wrt e2} \\
&&  m^{\epsi_2}  Du^{\epsi_2}  \weakly m 
Du  \hbox{ weakly in } L^1(\Tt^d;\RR^d ) \hbox{
as } \epsi_2 \to 0^+, \label{meDue wrt e2'}\\
&& \sqrt{\epsi_2} u^{\epsi_2} \to 0 \hbox{ in }
W^{1,2}(\Tt^d) \hbox{ as } \epsi_2 \to
0^+, \label{sue wrt e2}\\
&&  \sqrt{\epsi_2} m^{\epsi_2}
\to 0 \hbox{ in } W^{1,2}(\Tt^d) \hbox{ as
} \epsi_2 \to
0^+, \label{sme wrt e2}\\
&&  \epsi_2 \Delta m^{\epsi_2}
\to 0 \hbox{ in } L^1(\Tt^d) \hbox{ as
} \epsi_2 \to
0^+, \label{D2me wrt e2} \\
&& h(\bfm^{\epsi_2}) \to h(\bfm) \hbox{
in }
L^1(\Tt^d) \hbox{ as } \epsi_2 \to
0^+. \label{hme wrt e2}
\end{eqnarray}}%

Moreover, using in addition the compact
embedding
\(BV(\Tt^d) \hookrightarrow L^{s}(\Tt^d)\)
for all  \(s\in [1, 1^{*})\), it follows from
\eqref{ue wrt e2}, \eqref{mea wrt e2}, \eqref{meaDue
wrt e2},  \eqref{D2me
wrt e2}, and \eqref{hme wrt e2} that, up to a further not relabeled subsequence of \((m^{\epsi_2}, u^{\epsi_2})_{\epsi_2}\), we have for \aev\ \(x\in \Tt^d\),
{\setlength\arraycolsep{0.5pt}
\begin{eqnarray}\label{pointwise e2}
\begin{aligned}
& u^{\epsi_2} (x) \to u(x), \quad m^{\epsi_2} (x) \to m(x), \quad
  (m^{\epsi_2})^{\frac{\alpha +1}{2}} (x) Du^{\epsi_2}
(x) \to m^{\frac{\alpha +1}{2}}(x) Du(x),   \\
& \epsi_2
\Delta m^{\epsi_2}(x) \to 0, \quad
h(\bfm^{\epsi_2})(x) \to h(\bfm)(x) \hbox{
as } \epsi_2 \to 0^+.
\end{aligned}
\end{eqnarray}}%

Finally, because \(\inf_{\Tt^d} m^\epsi >0\) and
\((m^\epsi)_{\epsi_1}\) converges \aev\ in \(\Tt^d\)
to \(m^{\epsi_2}\) as \(\epsi_1 \to 0^+\), we
conclude that \(m^{\epsi_2} \geq 0\) \aev\ in \(\Tt^d\). In turn, the latter condition together with the
second convergence in \eqref{pointwise e2} yields
\(m \geq 0\) \aev\ in \(\Tt^d\). This concludes
Step~1. 
\medskip

\textit{Step 2.} We claim that the pair \((m,u)\) with \(m \geq 0\) \aev\ in \(\Tt^d\) determined in Step~1 is a weak solution to \eqref{Pquad}.

The argument is the same as the one in the proof of Theorem~\ref{Thm:main}.

\medskip

\textit{Step 3.} We prove that the pair \((m,u)\)
with \(m \geq 0\) \aev\ in \(\Tt^d\) determined
in Step~1 satisfies \(\int_{\Tt^d} m \,\dx =1\),  $\div(m Du) +\sigma^2\Delta m  \in
L^{\frac{\alpha
+ 1}{2}2^*}(\Tt^d)$, and \eqref{72'}.


Let $\varphi \in C^{\infty} (\Tt^d)$. Multiplying
   \eqref{eqnue0} by \(\varphi\),
 integrating the resulting equality over \(\Tt^d\),
and using integration by parts, we obtain
{\setlength\arraycolsep{0.5pt}%
\begin{eqnarray*}
\begin{aligned}
&\int_{\Tt^d}(m^\epsi -1) \varphi\,\dx   + \epsi_1
\int_{\Tt^d} u^\epsi\varphi\,\dx  + \epsi_1
\int_{\Tt^d} \Delta^{p}u^\epsi \Delta^p \varphi\,\dx+ \epsi_2
\int_{\Tt^d} u^\epsi\varphi\,\dx  - \epsi_2
\int_{\Tt^d} u^\epsi \Delta \varphi\,\dx
\\ 
& \qquad=-  \int_{\Tt^d} m^\epsi Du^\epsi \cdot D\varphi
\,\dx+\sigma^2
\int_{\Tt^d} m^\epsi \Delta \varphi\,\dx.
\end{aligned}
\end{eqnarray*}}%

Letting  \(\epsi_1\to 0^+\) first and using 
\eqref{me wrt e1}, \eqref{sue wrt e1}, \eqref{ue wrt e1}, and \eqref{meDue wrt e1}, and then letting  \(\epsi_2\to 0^+\) and invoking  \eqref{mea wrt e2}, \eqref{sue wrt e2}, and \eqref{meDue wrt e2'}, we conclude that
{\setlength\arraycolsep{0.5pt}%
\begin{eqnarray*}
\begin{aligned}
\int_{\Tt^d}(m -1) \varphi\,\dx  
&=- \int_{\Tt^d} m Du \cdot D\varphi
\,\dx+\sigma^2
\int_{\Tt^d} m\Delta \varphi\,\dx \\
&= \langle \div(m Du)+\sigma^2\Delta m,\varphi \rangle_{\mathcal{D}'
(\Tt^d), C^\infty(\Tt^d)}.
\end{aligned}
\end{eqnarray*}}%
Because $m-1  \in L^{\frac{\alpha + 1}{2}2^*}(\Tt^d)$
and  $\varphi \in C^{\infty} (\Tt^d)$ is arbitrary,
it follows that $\div(m Du) +\sigma^2\Delta m  \in L^{\frac{\alpha
+ 1}{2}2^*}(\Tt^d)$ and \eqref{72'} holds \aev\
in \(\Tt^d\).

Finally, taking $\varphi =1$ in the equality above yields $\int_{\Tt^d} m \,\dx
=1$.
This completes Step~3.
\medskip

\textit{Step 4.} We prove that the pair \((m,u)\)
with \(m \geq 0\) \aev\ in \(\Tt^d\) determined
in Step~1 satisfies \eqref{72}.

Let $\varphi \in C^{\infty} (\Tt^d)$ with $\varphi
\geq 0$. Multiplying
   \eqref{eqnme0} by \(\varphi\),
 integrating the resulting equality over \(\Tt^d\),
and using the integration by parts formula and the condition $\beta_{\epsi_1}(\cdot)\leq
0$, we get
{\setlength\arraycolsep{0.5pt}%
\begin{eqnarray}
\begin{aligned} \label{72 e2}
&\int_{\Tt^d} -u^\epsi
\varphi\,\dx + \intt V(x, m^\epsi, h(\bfm^\epsi)) \varphi\,\dx 
+ \sigma^{2} \int_{\Tt^d} u^\epsi
\Delta \varphi\, \dx + \epsi_1
\int_{\Tt^d} m^\epsi\varphi\,\dx  
\\ 
&\qquad+ \epsi_1
\int_{\Tt^d} \Delta^{p}m^\epsi \Delta^p \varphi\,\dx+ \epsi_2
\int_{\Tt^d} m^\epsi\varphi\,\dx
 - \epsi_2
\int_{\Tt^d} m^\epsi \Delta \varphi\,\dx
\\
&\qquad=
\int_{\Tt^d}\frac{|Du^\epsi|^2}{2} \varphi\,\dx
- \int_{\Tt^d} \beta_{\epsi_1}(m^\epsi) \varphi\,\dx
\geq 
\int_{\Tt^d}\frac{|Du^\epsi|^2}{2} \varphi\,\dx.
\end{aligned}
\end{eqnarray}}%
 First, we let  \(\epsi_1\to 0^+\) and then  \(\epsi_2\to 0^+\).
According to \eqref{ue wrt e1}, \eqref{me wrt e1}, \eqref{sme wrt e1}, \eqref{hme wrt e1},
\eqref{pointwise e1}, \eqref{ue wrt e2}, \eqref{mea
wrt e2}, \eqref{sme wrt e2}, \eqref{hme wrt e2}, \eqref{pointwise e2},    the sequential
lower semicontinuity of the  convex functional 
{\setlength\arraycolsep{0.5pt}%
\begin{eqnarray*}
\begin{aligned}
w\in L^2(\Tt^d;\RR^d) \mapsto 
\int_{\Tt^d}\frac{|w(x)|^2}{2}
\varphi(x)\,\dx \in \RR^+_0
\end{aligned}
\end{eqnarray*}}%
with respect to the weak convergence
in \(L^2(\Tt^d;\RR^d)\), and  the continuity of $V$ together with Fatou's lemma and Assumption~\eqref{V:Fatou}, we deduce that
{\setlength\arraycolsep{0.5pt}%
\begin{eqnarray*}
\begin{aligned}
&\int_{\Tt^d} -u
\varphi\,\dx + \intt V(x,m ,h(\bfm)) 
\varphi\,\dx + \sigma^{2} \int_{\Tt^d} u \Delta \varphi\, \dx \geq \int_{\Tt^d}
\frac{|Du|^2}{2}
\varphi\,\dx.
\end{aligned}
\end{eqnarray*}}%
Equivalently,
{\setlength\arraycolsep{0.5pt}%
\begin{eqnarray*}
\begin{aligned}
&\Big\langle-u - \frac{|Du|^2}{2}
+ V(x,m,h(\bfm)) + \sigma^2 \Delta
u,\varphi \Big \rangle_{\mathcal{D}'
(\Tt^d), C^\infty(\Tt^d)} \geq0,
\end{aligned}
\end{eqnarray*}}%
and so, \eqref{72}  holds since
$\varphi \in C^{\infty} (\Tt^d)$
with $\varphi
\geq 0$ is arbitrary.
\end{proof}

Next, we address the deterministic case, $\sigma=0$, and improve the supersolution property from Proposition \ref{supersolprop}.

\begin{corollary}\label{Cor:quad}
Suppose that Assumptions~\eqref{h}, \eqref{hh}, \eqref{g},
and \eqref{CC} hold with \(\sigma = 0\). If \(d\geq8\), assume further
that \(\alpha > \frac{d-4}{2} \) in Assumption~\eqref{CC}.  Then,
the weak solution, \((m,u)\), of \eqref{Pquad} given
by  Theorem~\ref{Thm:quad} satisfies
\begin{eqnarray*}
\begin{aligned}
\begin{cases}
\Big(\displaystyle-u- \frac{|Du|^2}{2} + V(x, m, h(\bfm))\Big)m = 0,\\
m - \div(m Du)  -1  =0,
\end{cases}
\end{aligned}
\end{eqnarray*}
\aev\ in \(\Tt^d\). 
\end{corollary}

\begin{proof}
Assume that \(\sigma=0\). 
In view of Theorem \ref{Thm:quad}, we are left to prove that
{\setlength\arraycolsep{0.5pt}
\begin{eqnarray}\label{HJ ae}
\Big(\displaystyle-u- \frac{|Du|^2}{2}
+ V(x,m, h(\bfm))
\Big)m = 0 \enspace \hbox{
\aev\
in \(\Tt^d\).}
\end{eqnarray}}%

In what follows, either \(d\leq 7\) and \(\alpha>0\)
or \(d\geq8\) and \(\alpha > \frac{d-4}{2} \).
Let \(\epsi = (\epsi_1,
\epsi_2) \in (0,1)^2\) and $(m^\epsi, u^\epsi) \in
C^\infty(\Tt^d)\times C^\infty(\Tt^d)$ with
$\inf_{\Tt^d}m^\epsi>0$
 be as in the proof of Theorem \ref{Thm:quad}. The arguments 
in the Step~4 of that proof show that by letting
\(\epsi_1\to 0^+\) in \eqref{72 e2},  we  have
\begin{equation*}
\int_{\Tt^d} \Big(- u^{\epsi_2 } - \frac{|Du^{\epsi_2}|^2}{2}
   +  V(x,m^{\epsi_2},h(\bfm^{\epsi_2}))+
\epsi_2
m^{\epsi_2}  - \epsi_2
 \Delta m^{\epsi_2} \Big) 
 \varphi\,\dx \geq 0
 \end{equation*}
for all $\varphi \in C^{\infty} (\Tt^d)$, with
\(\varphi \geq 0\). Hence, because \(m^{\epsi_2} \geq 0\) \aev\ in \(\Tt^d\),
we conclude that
\begin{equation}
\label{HJ ae geq}
\Big(- u^{\epsi_2 } - \frac{|Du^{\epsi_2}|^2}{2}
   +  V(x,m^{\epsi_2},h(\bfm^{\epsi_2}))+
\epsi_2
m^{\epsi_2}  - \epsi_2
 \Delta m^{\epsi_2} \Big) m^{\epsi_2}
\geq 0
\end{equation}
\aev\ in \(\Tt^d\).

Next, we prove the converse inequality. Let $\varphi \in C^{\infty} (\Tt^d)$ with \(\varphi \geq 0\). Multiplying   \eqref{eqnme0}  by \(m^\epsi\varphi\),
 integrating the resulting equality
over \(\Tt^d\),
and using integration by parts,
we get
{\setlength\arraycolsep{0.5pt}%
\begin{eqnarray}
&& - \int_{\Tt^d} u^\epsi m^\epsi\varphi
\,\dx - 
\int_{\Tt^d}\frac{|Du^\epsi|^2}{2}
 m^\epsi \varphi\,\dx
  + \int_{\Tt^d}  V(x,  m^\epsi,h(\bfm^\epsi))m^\epsi\varphi
\,\dx  
\nonumber \\
 &&\qquad+ \epsi_1
\int_{\Tt^d} (m^\epsi)^2\varphi\,\dx
 + \epsi_1
\int_{\Tt^d} \Delta^{p}m^\epsi
\Delta^p (m^\epsi\varphi)\,\dx
 +  \int_{\Tt^d} \beta_{\epsi_1}(m^\epsi)
m^\epsi \varphi\,\dx \nonumber\\
&&\qquad + \epsi_2
\int_{\Tt^d} (m^\epsi)^2 \varphi\,\dx
 - \epsi_2
\int_{\Tt^d} \Delta m^\epsi m^\epsi  \varphi\,\dx
=0.\label{HJ me}
\end{eqnarray}}

Using \eqref{ue wrt e1} and \eqref{me wrt e1} together with the embeddings
mentioned at the beginning of Step~1 of the previous
proof,
we have 
{\setlength\arraycolsep{0.5pt}%
\begin{eqnarray}
\label{HJ lim e11}
&&\lim_{\epsi_1 \to 0^+} \bigg(
- \int_{\Tt^d} u^\epsi m^\epsi\varphi
\,\dx     + \epsi_2
\int_{\Tt^d} (m^\epsi)^2 \varphi\,\dx
 - \epsi_2
\int_{\Tt^d} \Delta m^\epsi m^\epsi
 \varphi\,\dx \bigg)\nonumber \\
 &&\qquad = - \int_{\Tt^d} u^{\epsi_2} m^{\epsi_2}\varphi
\,\dx   
  + \epsi_2
\int_{\Tt^d} (m^{\epsi_2})^2 \varphi\,\dx
 - \epsi_2
\int_{\Tt^d} \Delta m^{\epsi_2} m^{\epsi_2}
 \varphi\,\dx. 
\end{eqnarray}}%

From Assumptions~\eqref{V:lb}, \eqref{G:sumg},
and \eqref{V:m^a+1}, it follows that
\begin{equation*}
V(x, m^{\epsi}, h(\bfm^{\epsi})) \geq \frac{1}{\kappa_2}
(m^\epsi)^\alpha - \kappa_2,
\end{equation*}
for some positive constant, \(\kappa_2\), independent
of \(\epsi\). Thus, because \(m^\epsi \varphi
\geq 0\) \aev\ in \(\Tt^d\), Fatou's lemma, the
continuity of \(V\), and the convergences \eqref{mea wrt e1} and \eqref{pointwise e1} imply that
\begin{equation}
\label{HJ lim e2'}
  \int_{\Tt^d}  V(x, m^{\epsi_2}, h(\bfm^{\epsi_2}))
 m^{\epsi_2}\varphi
\,\dx \leq \liminf_{\epsi_1 \to 0^+}  \int_{\Tt^d}  V(x,
m^{\epsi}, h(\bfm^{\epsi}))
 m^\epsi\varphi
\,\dx .
\end{equation}

By definition, \(\beta_{\epsi_1}(\cdot)
\leq 0\) and \(\beta_{\epsi_1}(s)
= 0\) if $s\geq \epsi_1$, and thus
{\setlength\arraycolsep{0.5pt}%
\begin{eqnarray*}
&&\bigg|  \int_{\Tt^d} \beta_{\epsi_1}(m^\epsi)
m^\epsi \varphi\,\dx \bigg| = \bigg|  \int_{\{x\in
\Tt^d\!:\, m^\epsi(x) \leq \epsi_1\}} \beta_{\epsi_1}
(m^\epsi)
m^\epsi \varphi\,\dx \bigg| \leq \epsi_1 \Vert
\varphi \Vert_{L^\infty(\Tt^d)} C,
\end{eqnarray*}}%
where \(C\) is the constant given by 
Lemma~\ref{Lem:moreape}. Hence,
{\setlength\arraycolsep{0.5pt}%
\begin{eqnarray}
\label{HJ lim e12}
&& \lim_{\epsi_1 \to 0^+} \bigg(   \epsi_1
\int_{\Tt^d} (m^\epsi)^2\varphi\,\dx
 + \epsi_1
\int_{\Tt^d} \Delta^{p}m^\epsi
\Delta^p (m^\epsi\varphi)\,\dx
 +\int_{\Tt^d} \beta_{\epsi_1}(m^\epsi)
m^\epsi \varphi\,\dx  
\bigg)
=0,\qquad
\end{eqnarray}}%
where we also used \eqref{sme wrt e1}.

Next, we prove that
{\setlength\arraycolsep{0.5pt}%
\begin{eqnarray}\label{HJ lim e13}
&&\lim_{\epsi_1 \to 0^+}  
\int_{\Tt^d}\frac{|Du^\epsi|^2}{2}
 m^\epsi \varphi\,\dx = \int_{\Tt^d}
 \frac{|Du^{\epsi_2}|^2}{2}
 m^{\epsi_2} \varphi\,\dx.
\end{eqnarray}}%
In view of the embeddings mentioned at the beginning
of Step~1 of the previous proof, we have 
{\setlength\arraycolsep{0.5pt}%
\begin{eqnarray}
&&m^\epsi \to m^{\epsi_2} \hbox{ 
in } L^p(\Tt^d) \hbox{ as } \epsi_1 \to
0^+ \hbox{ and  for all } p\in (1,(2^*)^*), \label{mdss} \\
&&m^\epsi \to m^{\epsi_2} \hbox{ 
in } L^r(\Tt^d) \hbox{ as } \epsi_1 \to
0^+ \hbox{ and  for all } r\in \Big ( 1, \frac{(\alpha+1)2^*}{2}\Big), \label{mads}
\\
&& |Du^\epsi|^2 \to |Du^{\epsi_2}|^2 \hbox{   in } L^{s}(\Tt^d) \hbox{ as } \epsi_1 \to 0^+\hbox{ and  for all } s\in \Big (1,\frac{2^*}{2}\Big).
\label{Dusq}
\end{eqnarray}}%
If \(d \leq 7 \), then \((\frac{2^*}{2})' < (2^*)^*\).
Accordingly, for some $p<\frac{2^*}{2}$ in \eqref{mdss} and some $s<\frac{2^*}{2}$ in \eqref{Dusq}, we get  $|Du^{\epsi}|^2  m^{\epsi}\to |Du^{\epsi_2}|^2  m^{\epsi_2}$ strongly in  $L^q(\Tt^d)$ with \(\frac{1}{q}
= \frac{1}{r} + \frac{1}{s}\) and $q>1$. 
Moreover, if  \(d\geq 8\) and \(\alpha > \frac{d-4}{2} \),
then there is \(q>1\) such that \(\frac{1}{q}
= \frac{1}{r} + \frac{1}{s}\), where \(r\) and
\(s\) are as in \eqref{mads} and \eqref{Dusq},
respectively. Therefore, we get again $|Du^{\epsi}|^2  m^{\epsi}\to |Du^{\epsi_2}|^2
 m^{\epsi_2}$ strongly in  $L^q(\Tt^d)$.

%
Finally, by \eqref{HJ me}, \eqref{HJ lim e11},
\eqref{HJ lim e2'}, \eqref{HJ lim e12}, and \eqref{HJ lim e13}, we
have
{\setlength\arraycolsep{0.5pt}%
\begin{eqnarray*}
&& \int_{\Tt^d} \Big(- u^{\epsi_2 } - \frac{|Du^{\epsi_2}|^2}{2}
 +  V(x,m^{\epsi_2},h(\bfm^{\epsi_2}))+ \epsi_2
m^{\epsi_2}  - \epsi_2
 \Delta m^{\epsi_2} \Big) m^{\epsi_2}
 \varphi\,\dx \leq 0.
\end{eqnarray*}}%
Because \(\varphi \in C^\infty(\Tt^d)\) with \(\varphi
\geq 0\) is arbitrary,
recalling \eqref{HJ ae geq}, it follows that
{\setlength\arraycolsep{0.5pt}%
\begin{eqnarray}\label{HJae e2}
&& \Big(- u^{\epsi_2 } - \frac{|Du^{\epsi_2}|^2}{2}
  +V(x,m^{\epsi_2},h(\bfm^{\epsi_2}))+ \epsi_2
m^{\epsi_2}  - \epsi_2
 \Delta m^{\epsi_2} \Big) m^{\epsi_2} = 0 \enspace
 \hbox{  \aev\ in } \Tt^d.\qquad
\end{eqnarray}}%

Let \(A:= \{ x \in \Tt^d \!: \, \eqref{pointwise e2} \hbox{ and \(m(x) \geq 0\) hold}\}\). Then, \(\Ll^d(\Tt^d \backslash
A) =0\). 

If \(x \in A\) is such that \(m(x) =0\), then clearly
{\setlength\arraycolsep{0.5pt}%
\begin{eqnarray}\label{HJae1}
&&\Big(\displaystyle-u(x)- \frac{|Du(x)|^2}{2}
+ V(x,m(x),h(\bfm)(x)) \Big)m(x) = 0.
\end{eqnarray}}%

If \(x \in A\) is such that \(m(x) >0\), then \(m^{\epsi_2}(x)>0\)
for all sufficiently small  \(\epsi_2\); moreover,
{\setlength\arraycolsep{0.5pt}%
\begin{eqnarray*}
\lim_{\epsi_2 \to 0^+} |Du^{\epsi_2}(x)|^2 m^{\epsi_2}(x)
&&= \lim_{\epsi_2 \to 0^+} |Du^{\epsi_2}(x)|^2
(m^{\epsi_2})^{{\alpha + 1}}(x) (m^{\epsi_2})^{{1-\alpha }}(x)\\
&& = |Du(x)|^2 m(x),
\end{eqnarray*}}%
which, together with \eqref{pointwise e2} and \eqref{HJae
e2}, proves that \eqref{HJae1}  holds for all such
$x$. This completes the proof
of \eqref{HJ ae}.
\end{proof}

%

\section{Final remarks}
\label{final}

The monotonicity method developed here is a powerful and flexible way to study 
MFGs. Several extensions can be considered using similar ideas, including
the standard stationary MFG:
\[
\begin{cases}
H(x, Du, D^2 u, m,h(\bfm))
=\Hh,\\
 - \div(m D_p H(x, Du, D^2 u,m, h(\bfm)))
+ \big(m D_{M_{ij}} H (x, Du, D^2 u, m,h(\bfm)) \big)_{x_i
        x_j} =0. 
\end{cases}
\]
In addition, there are several areas where future developments are likely. 
First, we foresee improvements in the regularity theory 
for weak solutions. Additional regularity is essential  to
prove the uniqueness of solutions. 
Second, the congestion problems examined here
may enjoy further regularity properties;  this topic
 was not explored in the present paper. 
Third, time-dependent MFGs
are a natural application of our methods. Here, 
we need to develop a 
different regularization method  because of the initial-terminal boundary conditions. 
Moreover, the regularity theory for these problems may differ substantially 
from the stationary case. Finally, our results
may   be of independent interest in the  calculus of variations 
as certain MFGs are the Euler-Lagrange equation of integral functionals. These functionals are convex but, in many cases, non-coercive,
and their study presents substantial challenges; see, for example, \cite{GM}.

\bibliographystyle{plain}

\bibliography{fgVarIneq}

\end{document}